\documentclass[11pt]{article}

\usepackage[english]{babel} 
\usepackage{amsmath, amssymb, amsfonts, amsthm} 
\usepackage{microtype} 
\usepackage[font=small,labelfont=sc]{caption} 
\usepackage{xcolor} 
\usepackage{enumitem} 
\usepackage{tikz} 
\usepackage{pgfplots} 
\pgfplotsset{compat=1.17} 
\usepackage{hyperref}
\usepackage{multirow}
\usepackage{longtable}

\hypersetup{ colorlinks=true, linkcolor=blue!50!black, citecolor=blue!50!black, urlcolor=blue!50!black }

\theoremstyle{plain} 
\newtheorem{theorem}{Theorem}[section] 
\newtheorem{proposition}[theorem]{Proposition} 
\newtheorem{lemma}[theorem]{Lemma} 
\newtheorem{corollary}[theorem]{Corollary} 
\newtheorem{conjecture}[theorem]{Conjecture}

\theoremstyle{definition} 
\newtheorem{definition}[theorem]{Definition} 
\newtheorem{remark}[theorem]{Remark}

\newcommand{\K}{\textsc{King}} 
\newcommand{\N}{\textsc{Knight}} 
\newcommand{\B}{\textsc{Bishop}} 
\newcommand{\R}{\textsc{Rook}} 
\newcommand{\AB}{\textsc{Archbishop}} 
\newcommand{\Q}{\textsc{Queen}} 
\newcommand{\Am}{\textsc{Amazon}} 
\newcommand{\C}{\textsc{Centaur}} 
\newcommand{\Em}{\textsc{Empress}} 
\newcommand{\Bee}{\textsc{Beeshop}} 
\newcommand{\Kook}{\textsc{Kook}} 
\newcommand{\Nork}{\textsc{Nork}} 
\newcommand{\AK}{\textsc{Pope}} 
\newcommand{\Z}{\mathbb{Z}} 
\newcommand{\standard}{\mathsf{Std}}
\newcommand{\str}{\mathrm{Str}} 
\newcommand{\Army}{\mathcal{A}} 
\newcommand{\Pieces}{\mathcal{P}}
\newcommand{\multiset}[2]{\left(\!\!\binom{#1}{#2}\!\!\right)}

\title{The Arithmetic of Chess Piece Strength on the $n \times n$ Board} 
\author{Frank M.\,V.\,Feys} 
\date{\today}

\begin{document} \maketitle

\begin{abstract} On the $n \times n$ chessboard, the move totals of distinct pieces turn out to satisfy a small number of striking arithmetic identities.
The total diagonal mobility of the \B{} and the total $8$-neighbor mobility of the \K{} are exactly proportional, with constant $n/12$, valid for every $n$.
Among nontrivial boards, the strengths of two distinct pieces drawn from a natural thirteen-piece alphabet coincide only for $n \in \{6, 8, 12\}$, with $n = 6$ and $n = 12$ arising as integer specializations of the bishop-king proportionality, and $n = 8$ arising from a separate cubic identity that singles out the standard chessboard.
This paper develops the systematic theory behind these phenomena.
We define the \emph{strength} of a chess piece $P$ on the $n \times n$ board as the probability that a uniformly random ordered pair of distinct squares forms a legal move of $P$ on the empty board, and prove four main results.
(1) An asymptotic dichotomy classifies pieces into \emph{riders} ($\Theta(1/n)$ strength) and \emph{leapers} ($\Theta(1/n^2)$ strength), with explicit rational leading constants.
(2) A stable-ordering theorem identifies the threshold $n^* = 24$ beyond which the strength order of the thirteen pieces becomes fixed, together with a complete tabulation of every transition for $4 \leq n \leq 24$.
(3) A complete classification of strength coincidences shows that they occur only at the three magic boards $n \in \{6, 8, 12\}$, accompanied by the closed-form identity $\str(\K) - \str(\N) = 12/(n^2(n + 1))$ valid at every $n$, the unique near-coincidence between the \B{} and the \N{} at $n = 10$ (gap $0.0606\%$), and the bishop-king proportionality $\str(\B)/\str(\K) = n/12$.
(4) A \emph{Strength Algebra Theorem} expresses the strength of any compound army as a linear functional of a four-dimensional ``atomic'' vector, and confines strength coincidences between distinct single pieces of the alphabet to the three magic boards.
As immediate consequences we obtain explicit strength-preserving single-piece substitution rules on each magic board, and a characterization of the $8 \times 8$ board as the unique nontrivial board on which the \R{} attains a strength matched by another piece in the alphabet (namely the \AB{}).
\end{abstract}

\section{Introduction}

The research presented in this paper began with the following innocent-looking question:

\begin{quote}
\emph{If, on an empty chessboard, one draws completely at random an arrow from one square to another, 
what is the probability that the move corresponding to that arrow can be made by some basic chess piece?}
\end{quote}

On the standard $8 \times 8$ board, the answer is $1792/4032 = 4/9 \approx 44.44\%$.
The numerator $1792$ is the total number of legal moves, summed across all sixty-four starting squares, of the \Am{} (a hypothetical fairy-chess piece combining the moves of the \Q{} and the \N{}, which subsumes all moves available to the basic chess pieces); the denominator $4032 = 64 \cdot 63$ is the number of ordered pairs of distinct squares.
We shall arrive at this number in Section~\ref{sec:n8} below.

The natural refinement of this question, and the one we pursue throughout the paper, is to ask the same probability for a fixed individual piece $P$:

\begin{quote}
\emph{If, on an empty $n \times n$ chessboard, one draws an arrow uniformly at random  from one square to another, what is the probability that the move it represents is a legal $P$-move?}
\end{quote}

Pursuing this question across all pieces and all board sizes in fact uncovers a remarkably rigid arithmetic structure.
A first hint is the following.
On the standard $8 \times 8$ chessboard, the \R{} has $896$ legal moves in total, summed across all sixty-four starting squares.
The same is true of the \AB{}, a hypothetical fairy-chess piece combining the moves of the \B{} and the \N{}.
The two totals are equal, exactly.
Equivalently, the probability that a uniformly chosen ordered pair of distinct squares on the standard chessboard 
forms a legal \R-move is $896/4032 = 22.22\%$, the same as the probability that it forms a legal \AB-move.
This coincidence is not robust: it does not hold on the $7 \times 7$ board, on the $9 \times 9$ board, or on any other nontrivial board size.
This paper is concerned with such coincidences and the algebraic structure that produces them.

We formalize the question above by defining the \emph{strength} of a 
chess piece $P$ on the $n \times n$ board as the probability that a uniformly 
random ordered pair of distinct squares $(p, q)$ has the property that $q$ is 
reachable from $p$ in one legal move of $P$ on the empty board.
This is, perhaps, the simplest possible mathematical model of the relative power of a chess piece.
Indeed, the model ignores essentially everything that makes chess a game.
Blocking, capturing, position-dependence, and dynamics in general 
 are completely absent.
On the standard $8 \times 8$ board, the model assigns the \N{} a strength higher than that of the \B{}, although standard chess intuition holds them roughly equal in value, with the \B{} usually 
considered slightly stronger.
And yet the simplicity is precisely the point.
The model captures only the geometric fact of how many squares each piece reaches on average, 
normalized to lie in $[0, 1]$, and this minimal information turns out to be enough to support a surprisingly rich theory.

On the standard $8 \times 8$ board, the strengths of the five standard chess pieces and several common fairy-chess compounds work out as follows.
Among the $4032$ ordered pairs of distinct squares, exactly $1792$ form a legal move of \emph{some} piece in our collection.
The strongest piece in this collection is the \Am{}, a hypothetical compound of the \Q{} and the \N{}, attaining a strength equal to $1792/4032 = 4/9 \approx 44.44\%$.
The \R-versus-\AB{} coincidence at $22.22\%$, from above, is one of several arithmetic features of this list of strengths.
The aim of the present paper is to identify all such features as functions of the board size $n$.

The question of how many squares each piece reaches, despite its elementary character, 
has not, to our knowledge, been systematically studied as a function of the board size $n$.
The substantial combinatorial literature on the chessboard concerns instead \emph{configuration-counting} questions, such as the number of non-attacking placements of pieces: examples include the $n$-queens problem and its generalizations, treated systematically in the $q$-Queens series of Chaiken, Hanusa, and Zaslavsky~\cite{ChaikenHanusaZaslavsky} via Ehrhart-quasi-polynomial methods, related work on safe-square counts in Miller, Sheng, and Turek~\cite{MillerShengTurek2021} and in Cashman, Cooper, Marquez, Miller, and Shuffelton~\cite{CashmanEtAl2024}, and Watkins's monograph~\cite{Watkins2004} for broader exposition.
The simpler \emph{move-counting} quantities $T_P(n)$ underlying our strength function are themselves Ehrhart-style integer-point counts in dilated rational polytopes attached to each piece's move set, and have been discussed informally in the chess-variants community without systematic study.\footnote{The chess totals $T_P(n)$ have been discussed informally in chess-variants writing; see \cite{Betza} for an early example.} 
What the present paper aims to add is a systematic mathematical theory of these counts and their ratios.

\paragraph{Main results.} 
Four main phenomena emerge when one studies the strength function across all $n$.

The first phenomenon is an \emph{asymptotic dichotomy}.
As $n \to \infty$, every piece's strength tends to zero, but the rates of decay split sharply into two classes.
\emph{Riders} (\R, \B, \Q, and their compounds) have strength $\Theta(1/n)$; \emph{leapers} (\K, \N, and their pure compounds) have strength $\Theta(1/n^2)$.
The leading constants are, in fact, simple rational numbers (Theorem~\ref{thm:asymptotics}).
Consequently, for sufficiently large $n$, the strength order is fixed; from $n = 24$ onward,
  the order of strengths of the thirteen pieces we consider is a fixed permutation (Theorem~\ref{thm:stable-ordering}).

The second is a \emph{complete classification of strength coincidences}.
Among nontrivial boards $n \geq 4$, the equation $\str(P) = \str(Q)$ for distinct pieces $P, Q$ in 
our alphabet has a solution if and only if $n \in \{6, 8, 12\}$ (Theorem~\ref{thm:coincidences}).
We refer to these three values as the \emph{magic boards}.
Two of them, $n = 6$ and $n = 12$, are integer specializations of a single remarkable algebraic relation $\str(\B) = (n/12)\str(\K)$ 
(Theorem~\ref{thm:bishop-king-proportionality}); the third, $n = 8$, comes from a separate identity involving the \R{}.

The third phenomenon is a small handful of \emph{exact closed-form identities} between specific pieces, valid for every $n$.
The most striking is the bishop-king proportionality
\begin{equation*}
\str(\B) \,=\, \frac{n}{12} \, \str(\K),
\qquad \text{equivalently,} \qquad
12 \, T_\B(n) \,=\, n \, T_\K(n),
\end{equation*}
an exact algebraic identity between a rider and a leaper, with no error term, valid for every integer $n \geq 2$ (Theorem~\ref{thm:bishop-king-proportionality}).
The proportionality constant $n/12$ takes integer values precisely at the divisors of $12$, and this is the structural mechanism behind two of the three magic boards: 
at $n = 12$ the constant becomes $1$ and one obtains $\str(\K) = \str(\B)$, while at $n = 6$ it becomes $1/2$ and one obtains $\str(\K) = \str(\B\B)$ for the double bishop.
Two further closed-form identities deserve mention.
The \K{} and the \N{} are leapers in the same asymptotic class, and their strengths satisfy $\str(\K) - \str(\N) = 12/(n^2(n+1))$ at every $n \geq 2$ (Theorem~\ref{thm:king-knight}), from which the \K{} dominates the \N{} uniformly.
The \B{} and the \N{}, despite belonging to different asymptotic classes, never satisfy $\str(\B) = \str(\N)$ at any integer $n$, but they come closest at $n = 10$ where the gap is a mere $0.0606\%$ (Theorem~\ref{thm:bishop-knight-min}).

The fourth is a structural theorem for the \emph{strength algebra} of compound armies.
The strength of a multi-piece army, that is, a multiset of pieces drawn from our alphabet, is determined entirely by a four-dimensional vector counting how many times each of the four ``atomic'' pieces $\K, \N, \B, \R$ appears in the army, summed across all its compound pieces.
The strength function is then a linear pairing of this atomic vector against a board-dependent vector $T(n) \in \mathbb{Q}^4$, and the integer differences $a(P) - a(Q)$ between atomic vectors of distinct single pieces $P, Q$ in our alphabet that annihilate this pairing form a single-piece kernel, trivial except on the three magic boards, where it is generated by an explicit small integer direction (Theorem~\ref{thm:strength-algebra}).
On these boards, and only on these, strength-preserving single-piece substitution rules exist between pieces in our alphabet.
As an immediate corollary, the standard chess army has exactly two strength-preserving variants on the $8 \times 8$ board obtained by single-piece substitutions, both involving \AB{}-for-\R{} swaps (Theorem~\ref{thm:standard-variants}).

\paragraph{The structural status of $n = 8$.} 
The four phenomena above interact in a way that distinguishes the $8 \times 8$ board within our model.
The bishop-king proportionality identifies the magic boards $n = 6$ and $n = 12$ as inevitable consequences of a single linear relation between two basic pieces.
The third magic board $n = 8$, by contrast, arises from a separate cubic identity involving the \R{}, and is the unique nontrivial board on which the \R{} attains a strength matched by some other piece in our collection (Theorem~\ref{thm:8-rook}).
We do not claim that this resolves the historical question of why $8 \times 8$ became the standard size of the chessboard; we record the result, however, as a small mathematical observation that the standard size occupies a structurally special position in the model.

\paragraph{Reading the paper.} 
Figure~\ref{fig:strengths}, in Section~\ref{sec:asymptotics}, plots the strengths of the principal pieces against $n$ on a log-log scale.
The two visible slopes correspond to the two asymptotic classes, the magic-board crossings appear as coincidences at $n = 6, 8, 12$, and the stable ordering emerges by $n = 24$.
The reader may find this figure useful when reading the technical sections.

\paragraph{Outline.} 
Section~\ref{sec:n8} gives the explicit strength values for $n = 8$ as a motivating special case, with the rook-archbishop coincidence explicit in the table.
Section~\ref{sec:formulas} derives closed-form formulas for the strength of each standard and compound piece, valid for general $n$.
Section~\ref{sec:asymptotics} proves the rider-leaper dichotomy and the stable-ordering theorem.
Section~\ref{sec:coincidences} proves the magic-board classification of strength coincidences, identifies the unique near-coincidence between the \B{} and the \N{} at $n = 10$, establishes the closed-form identity $\str(\K) - \str(\N) = 12/(n^2(n + 1))$ valid for  every $n$, and proves the bishop-king proportionality $\str(\B) = (n/12) \str(\K)$, from which two of the three magic boards arise as integer specializations.
Section~\ref{sec:signatures} tabulates the complete sequence of changes in the strength order of the thirteen pieces as $n$ varies.
Section~\ref{sec:algebra} develops the atomic-vector framework, proves the Strength Algebra Theorem, illustrates it with a worked computation of the standard army's strength on different boards (Section~\ref{sec:worked-example}), and identifies the strength-preserving piece-substitution rules on each of the magic boards $n \in \{6, 8, 12\}$.
Section~\ref{sec:conclusion} discusses connections to chess variants, extensions, and open questions.
The structural theorems proved here, in particular the asymptotic dichotomy, the stable-ordering threshold,
 the magic-board classification, the bishop-king proportionality, and the Strength Algebra Theorem, are to our knowledge new.

\paragraph{Notation.}
We write $\mathbb{N} = \{0, 1, 2, \ldots\}$ for the set   of nonnegative integers, 
and $\mathbb{Z}_{>0} = \{1, 2, 3, \ldots\}$  for the positive integers.
The notation $[n]^2$ stands for $\{0, 1, \ldots, n - 1\}^2$, 
used throughout to index the cells of the $n \times n$ 
board by their integer coordinates.

\section{The Case $n = 8$} \label{sec:n8}

Let us begin with an explicit calculation on the standard board, both as motivation and as a tool to fix conventions.
The $n^2 = 64$ squares of the $8 \times 8$ board, together with the $n^2(n^2 - 1) = 4032$ ordered pairs of distinct squares, will be the foundation of all subsequent definitions.

We treat the five standard chess pieces other than the pawn ($\K$, $\N$, $\B$, $\R$, $\Q$), together with several common ``fairy chess'' compounds, namely the \AB{} (combination of $\B$ and $\N$, also known as the Hawk in Seirawan chess), the \Em{} (combination of $\R$ and $\N$, also known as the Elephant or Chancellor), the \C{} (combination of $\K$ and $\N$), and the \Am{} (combination of $\Q$ and $\N$).

For each piece $P$, let $T_P(n)$ denote the total number of legal $P$-moves across all starting squares on the $n \times n$ empty board.\footnote{Quantities of this type have been discussed informally in the chess-variants community; see~\cite{Betza}.} 
By direct case analysis on the standard board, summing the per-square mobilities listed in Figure~\ref{fig:n8} below, we find that
\begin{align*}
T_\K(8) &= 4 \cdot 3 + 24 \cdot 5 + 36 \cdot 8 = 420, \\
T_\N(8) &= 4 \cdot 2 + 8 \cdot 3 + 20 \cdot 4 + 16 \cdot 6 + 16 \cdot 8 = 336, \\
T_\B(8) &= \tfrac{1}{2}\bigl(28 \cdot 7 + 20 \cdot 9 + 12 \cdot 11 + 4 \cdot 13\bigr) = 280, \\
T_\R(8) &= 64 \cdot 14 = 896.
\end{align*}
The factor of $\tfrac12$ for the \B{} reflects our convention that ``the \B{}'' is a single-color bishop, confined to its starting color complex; the double bishop $\B\B$, equal in moves to a single bishop on each color complex, has $T_{\B\B}(n) = 2 T_\B(n)$ and $T_{\B\B}(8) = 560$.

Note that the compound pieces have totals equal to the sum of their components:
\begin{align*}
T_\Q(8) &= T_{\B\B}(8) + T_\R(8) = 560 + 896 = 1456, \\
T_\AB(8) &= T_{\B\B}(8) + T_\N(8) = 560 + 336 = 896, \\
T_\Em(8) &= T_\R(8) + T_\N(8) = 896 + 336 = 1232, \\
T_\C(8) &= T_\K(8) + T_\N(8) = 420 + 336 = 756, \\
T_\Am(8) &= T_\Q(8) + T_\N(8) = 1456 + 336 = 1792.
\end{align*}
The total count $T_\Am(8) = 1792$ is, in particular, the numerator of the answer $1792/4032 = 4/9$ to the motivating question with which the introduction opened.

\begin{figure}[h]
\centering
\renewcommand{\arraystretch}{1.0}
\footnotesize
\setlength{\tabcolsep}{3pt}
\begin{tabular}{ | c c | c c | c c | c c | c c | c c | c c | c c | }
\hline
 N\,2 & R\,14 & N\,3 & R\,14 & N\,4 & R\,14 & N\,4 & R  \,14 &
N\,4 & R\,14 & N\,4 & R\,14 & N\,3 & R\,14 & N\,2 & R\,14 \\
B\,7 & K\,3 & B\,7 & K\,5 & B\,7 & K\,5 & B\,7 & K\,5 &
B\,7 & K\,5 & B\,7 & K\,5 & B\,7 & K\,5 & B\,7 & K\,3 \\
\hline
N\,3 & R\,14 & N\,4 & R\,14 & N\,6 & R\,14 & N\,6 & R\,14 &
N\,6 & R\,14 & N\,6 & R\,14 & N\,4 & R\,14 & N\,3 & R\,14  \\
B\,7 & K\,5 & B\,9 & K\,8 & B\,9 & K\,8 & B\,9 & K\,8 &
B\,9 & K\,8 & B\,9 & K\,8 & B\,9 & K\,8 & B\,7 & K\,5 \\
\hline
N\,4 & R\,14 & N\,6 & R\,14 & N\,8 & R\,14 & N\,8 & R\,14 &
N\,8 & R\,14 & N\,8 & R\,14 & N\,6 & R\,14 & N\,4 & R\,14 \\
B\,7 & K\,5 & B\,9 & K\,8 & B\,11 & K\,8 & B\,11 & K\,8 &
B\,11 & K\,8 & B\,11 & K\,8 & B\,9 & K\,8 & B\,7 & K\,5 \\
\hline
N\,4 & R\,14 & N\,6 & R\,14 & N\,8 & R\,14 & N\,8 & R\,14 &
N\,8 & R\,14 & N\,8 & R\,14 & N\,6 & R\,14 & N\,4 & R\,14 \\
B\,7 & K\,5 & B\,9 & K\,8 & B\,11 & K\,8 & B\,13 & K\,8 &
B\,13 & K\,8 & B\,11 & K\,8 & B\,9 & K\,8 & B\,7 & K\,5  \\
\hline
N\,4 & R\,14 & N\,6 & R\,14 & N\,8 & R\,14 & N\,8 & R\,14 &
N\,8 & R\,14 & N\,8 & R\,14 & N\,6 & R\,14 & N\,4 & R\,14 \\
B\,7 & K\,5 & B\,9 & K\,8 & B\,11 & K\,8 & B\,13 & K\,8 &
B\,13 & K\,8 & B\,11 & K\,8 & B\,9 & K\,8 & B\,7 & K\,5 \\
\hline
N\,4 & R\,14 & N\,6 & R\,14 & N\,8 & R\,14 & N\,8 & R\,14 &
N\,8 & R\,14 & N\,8 & R\,14 & N\,6 & R\,14 & N\,4 & R\,14 \\
B\,7 & K\,5 & B\,9 & K\,8 & B\,11 & K\,8 & B\,11 & K\,8 &
B\,11 & K\,8 & B\,11 & K\,8 & B\,9 & K\,8 & B\,7 & K\,5  \\
\hline
N\,3 & R\,14 & N\,4 & R\,14 & N\,6 & R\,14 & N\,6 & R\,14 &
N\,6 & R\,14 & N\,6 & R\,14 & N\,4 & R\,14 & N\,3 & R\,14 \\
B\,7 & K\,5 & B\,9 & K\,8 & B\,9 & K\,8 & B\,9 & K\,8 &
B\,9 & K\,8 & B\,9 & K\,8 & B\,9 & K\,8 & B\,7 & K\,5 \\
\hline
N\,2 & R\,14 & N\,3 & R\,14 & N\,4 & R\,14 & N\,4 & R\,14 &
N\,4 & R\,14 & N\,4 & R\,14 & N\,3 & R\,14 & N\,2 & R\,14 \\
B\,7 & K\,3 & B\,7 & K\,5 & B\,7 & K\,5 & B\,7 & K\,5 &
B\,7 & K\,5 & B\,7 & K\,5 & B\,7 & K\,5 & B\,7 & K\,3  \\
\hline
\end{tabular}
\caption{Per-square mobility on the $8 \times 8$ board. 
Each square is divided into four sub-cells, each labeled by a piece and showing the
number of legal moves of that piece from this square: \emph{N} for the
\N{}, \emph{R} for the \R{}, \emph{B} for the double-\B{}
 (i.e., the sum of the moves of a light-square \B{} and a dark-square \B{} on this square), and \emph{K} for the \K{}. 
 The four sub-cells are arranged
spatially as
$\left(\begin{smallmatrix} N & R \\ B & K \end{smallmatrix}\right)$.}
\label{fig:n8}
\end{figure}

\begin{remark} 
The asymmetry between the \B{} and the other pieces will recur: all common chess pieces other than the \B{} act symmetrically with respect to the white-square/black-square color complexes, while a single \B{} is permanently confined to one color.
The \Q{} and the \AB{} therefore contain a $\B\B$, and hence two bishops in our atomic decomposition (formalized in Section~\ref{sec:atomic-decomp}).
\end{remark}

\begin{definition} \label{def:strength} 
Let $P$ be a chess piece 
and  $n \geq 1$.
The \emph{strength} of $P$ on the $n \times n$ board, denoted $\str(P)$,
 is the probability that a uniformly chosen ordered pair of distinct squares $(p, q)$ has the property that $q$ is reachable from $p$ in one legal move of $P$ on the empty board.
Equivalently,
\begin{equation}
\label{eq:str-def}
\str(P) \,:=\, \frac{T_P(n)}{n^2(n^2 - 1)},
\end{equation}
where $T_P(n)$ is the total number of legal $P$-moves on the $n \times n$ empty board, and $n^2(n^2 - 1)$ is the number of ordered pairs of distinct squares.
\end{definition}

For $n=8$, the denominator is $4032$, 
yielding the following strengths:
$$
\begin{array}{rcl}
\str(\B) = 280/4032 \approx 6.94\%,    & \quad & \str(\C)  = 756/4032 = 18.75\%,        \\
\str(\N) = 336/4032 \approx 8.33\%,    & \quad & \str(\AB) = 896/4032 \approx 22.22\%,  \\
\str(\K) = 420/4032 \approx 10.42\%,   & \quad & \str(\Em) = 1232/4032 \approx 30.56\%, \\
\str(\R) = 896/4032 \approx 22.22\%,   & \quad & \str(\Q)  = 1456/4032 \approx 36.11\%, \\
                                       & \quad & \str(\Am) = 1792/4032 \approx 44.44\%.
\end{array}
$$

\begin{remark} \label{rem:rook-archbishop} 
A first observation, namely $\str(\R) = \str(\AB)$ on the $8 \times 8$ board, deserves to be flagged at once.
This nontrivial coincidence is the simplest manifestation of the arithmetic structure that organizes the rest of the paper, and is the content of Proposition~\ref{thm:archbishop-rook} below.
We note that the coincidence arises without any reference to ``probability'' or ``strength'' whatsoever; 
   one need only count.
\end{remark}

\subsection{The Thirteen Pieces: A Dictionary} \label{sec:dictionary}

For ease of reference we collect the thirteen pieces studied in this paper into a single dictionary.
Each piece is recorded by its symbol, its move-set decomposition into basic pieces, and, where applicable, the alternative names under which the same piece appears in historical and contemporary chess variants.
The fairy-piece naming conventions in the chess-variants literature are notoriously non-standardized, and the same compound piece often appears under many different names; we record the most common ones for each piece, together with a few historically significant or geographically distinctive ones, drawing primarily on 
Pritchard~\cite{Pritchard2007}, 
the chess-variants encyclopedia~\cite{Piececlopedia}, 
and Dickins~\cite{Dickins1971}.

{\renewcommand{\arraystretch}{1.15}\footnotesize
\setlength{\LTpre}{6pt}\setlength{\LTpost}{6pt}
\begin{longtable}{| l | p{3.7cm} | p{9.5cm} |}
\hline
\textit{Piece} & \textit{Decomposition} & \textit{Alternative names and notes} \\
\hline
\endfirsthead
\hline
\textit{Piece} & \textit{Decomposition} & \textit{Alternative names and notes} \\
\hline
\endhead
\hline
\endfoot
\hline
\endlastfoot
\K{} & \K{} & Standard piece. 
The non-royal version is sometimes called the \emph{Mann} or \emph{commoner} \\
\N{} & \N{} & Standard piece \\
\B{} & \B{} (single color) & Standard piece. 
We treat the \B{} as confined to one color complex throughout \\
\R{} & \R{} & Standard piece \\ \hline
\Q{} & $\B\B + \R$ & Standard piece. 
Historical names include \emph{ferz} (in shatranj, with restricted move) and \emph{minister} \\
\AB{} & $\B\B  +  \N$ & The most prolific in alternative names. 
\emph{Princess} (standard among problemists, favored by Dickins~\cite{Dickins1971}), \emph{Cardinal} (Grand Chess of Freeling~\cite{Pritchard2007}), \emph{Archbishop} (Capablanca chess and Gothic Chess), \emph{Hawk} (Seirawan chess~\cite{harper2007seirawan}), \emph{Janus} (Janus Chess), \emph{Paladin} (Cavalier Chess), \emph{Vizir} (historically in Turkish Great Chess, not to be confused with the modern Wazir), \emph{Centaur} (Carrera's Chess of $1617$~\cite{Pritchard2007}; this conflicts with our use of Centaur for $\K + \N$, see below), \emph{Minister}; less common are \emph{Adjutant}, \emph{Equerry} (Bird's Chess of $1874$), \emph{Fox} (Wolf Chess), \emph{Davidson}, \emph{Deacon}, \emph{Horseman}, \emph{Monk}, \emph{Prime Minister}, \emph{Squire}, \emph{Templar} \\
\Em{} & $\R + \N$ & \emph{Empress} (standard among problemists, favored by Dickins~\cite{Dickins1971}), \emph{Chancellor} (Capablanca chess, Chancellor Chess of Foster~\cite{Pritchard2007}, Gothic Chess), \emph{Marshal} or \emph{Marshall} (Grand Chess; The Sultan's Game of Tressan), \emph{Elephant} (Seirawan chess~\cite{harper2007seirawan}),  \emph{Knook} (informal usage on Chess.com), \emph{Champion} (Carrera's Chess of $1617$, conflicting with the Champion of Omega Chess), \emph{Dabbabah} or \emph{war machine} (historically in Turkish Great Chess; this conflicts with the modern $(0,2)$-leaper called Dabbabah); less common are \emph{Concubine}, \emph{Cannon}, \emph{Colonel}, \emph{Duke}, \emph{Guard}, \emph{Lord Chancellor}, \emph{Samurai}, \emph{Superrook}, \emph{Tank} \\
\C{} & $\K + \N$ & \emph{Centaur} is the standard modern problemist name for $\K + \N$, although the term was originally used by Carrera in $1617$ for $\B + \N$ (our \AB). 
Also called \emph{Crowned Knight}, \emph{Mann + Knight}, \emph{Eques Rex} (Fusion Chess of Duniho); occasionally \emph{Paladin} (in some Cavalier Chess variants, again conflicting with \AB) \\
\Am{} & $\Q + \N$ & \emph{Amazon} is by far the most common modern name, but the piece was historically called the \emph{Giraffe} (Turkish Great Chess), the \emph{Maharajah} (Maharajah and the Sepoys), and the \emph{Omnipotent Queen}; other names include \emph{Dragon}, \emph{Terror}, \emph{General} (all in Dickins~\cite{Dickins1971}), \emph{Angel}, \emph{Commander}, \emph{Crown Prince}, \emph{Grand Chancellor}, \emph{Royal Guard}, \emph{Superqueen}, \emph{Wyvern} \\ \hline
\Bee{} & $\K + \B$ & \emph{Crowned Bishop} is the standard descriptive name. 
Also called \emph{Dragon Horse} (in Shogi, where this piece is the promoted Bishop, called \emph{ry\=uma} in Japanese); \emph{Pontiff} (in Fusion Chess of Duniho, where it was originally called \emph{Pope}, conflicting with our \AK{} below; the Fusion Chess piece was renamed to avoid this confusion in subsequent variants), \emph{Primate} (Charles Gilman). 
Some authors use \emph{Patriarch} or \emph{Deacon} \\
\AK{} & $\K + \AB$ & A ``king of Archbishops'', whence \emph{Pope}. 
The compound has no widely attested alternative name. 
Descriptive alternatives include \emph{Knighted Beeshop}, \emph{Crowned Archbishop}, or \emph{Crowned Princess}. 
Note that \emph{Pope} has been used in Fusion Chess for $\K + \B$ (our \Bee), creating a naming clash that motivates our specific use here for $\K + \N + \B$ \\
\Nork{} & $\K + \N + \R$ & Acronym from \textbf{N}ight + \textbf{R}ook + \textbf{K}ing (the three constituent pieces). 
The compound has no widely attested alternative name; descriptive alternatives include \emph{Knighted Kook}, \emph{Crowned Empress}, or \emph{Crowned Chancellor} \\
\Kook{} & $\K + \R$ & Sound-blend of \emph{K}ing and \emph{Rook}: K-(R)ook. 
Standardly called \emph{Crowned Rook}, or \emph{Dragon King} in Shogi, where this is the promoted Rook (called \emph{ry\=u\={o}} in Japanese). 
Also appears in Mongolian Shatar as the \emph{half-power baras} \\ \hline
\end{longtable}
}

The first four rows are the \emph{basic pieces},   namely $\K, \N, \B, \R$.
These are the pieces whose move sets cannot  be decomposed as a
   disjoint union of smaller standard-chess-piece move sets.
The next five rows list the standard fairy-chess compounds: the \Q{} (the only basic-pieces-only compound appearing in standard chess), the \AB{} and \Em{} (the two compounds introduced by Capablanca, also adopted by Seirawan under different names), and the two further compounds \C{} and \Am{} obtained by adding the \K{} or the \Q{} to the \N{}.
The final four rows are custom names introduced for this paper, denoting the four remaining union-of-basic-pieces compounds that occur in our analysis.
The double bishop $\B\B$, equivalent to two single-color bishops on opposite color complexes, is not among the thirteen members of $\Pieces$ but appears implicitly inside the \Q, the \Am, and the \AB.

A few naming clashes deserve explicit mention.
The name \emph{Centaur} is used in this paper for the $\K + \N$ compound (following modern problemist convention), but historically the same name was used by Carrera in $1617$ for the $\B + \N$ compound (our \AB).
The name \emph{Pope} is used here for the $\K + \B + \N$ compound, but the same name has been used in Fusion Chess for the $\K + \B$ compound (our \Bee), where it was later renamed \emph{Pontiff} to reduce confusion.
The name \emph{Vizir} appears historically for the $\B + \N$ compound (our \AB), but is not to be confused with the modern \emph{Wazir}, which denotes the $(1, 0)$-leaper.
The name \emph{Dabbabah} appears historically for the $\R + \N$ compound (our \Em), but is not to be confused with the modern \emph{Dabbabah}, which denotes the $(0, 2)$-leaper.
The non-standardization of fairy-piece names across the chess-variants literature is the central motivation for fixing the naming conventions used in this paper at the outset.

\section{Closed-Form Strength Formulas for General $n$} \label{sec:formulas}

We now compute $T_P(n)$ for general $n$, where $P$ ranges over the nine pieces introduced in Section~\ref{sec:n8}.
The pattern of per-square mobility on an $n \times n$ board is captured by partitioning the squares into \emph{shells}, that is, the outermost ring, the second-outermost ring, and so on.
Let $B_i$ denote the $i$-th shell.
Then $B_1$ consists of the $4(n - 1)$ squares on the outer perimeter, $B_2$ consists of the $4(n - 3)$ squares on the next ring, and in general $| B_i | = 4(n - 2i + 1)$ for $1 \leq i \leq \lfloor n/2 \rfloor$.
(When $n$ is odd, the central single square forms the innermost shell, with one element.)

\begin{theorem}
\label{thm:strengths}
Let $n \geq 2$. 
The move totals $T_P(n)$ on the $n \times n$ board of the thirteen pieces $P \in \Pieces$ are given by the following formulas:
\begin{center}
\renewcommand{\arraystretch}{1.3}
\begin{tabular}{| l | l |}
\hline
$P$ & $T_P(n)$ \\ \hline
$\K$ & $4(n - 1)(2n - 1)$ \\
$\N$ & $8(n - 1)(n - 2)$ \\
$\B$ & $\tfrac{1}{3} n(n - 1)(2n - 1)$ \\
$\R$ & $2(n - 1)n^2$ \\
$\Q$ & $\tfrac{2}{3} n(n - 1)(5n - 1)$ \\
$\AB$ & $\tfrac{2}{3}(n - 1)(2n^2 + 11n - 24)$ \\
$\Em$ & $2(n - 1)(n^2 + 4n - 8)$ \\
$\C$ & $4(n - 1)(4n - 5)$ \\
$\Am$ & $\tfrac{2}{3}(n - 1)(5n^2 + 11n - 24)$ \\
$\Bee$ & $\tfrac{1}{3}(n - 1)(2n - 1)(n + 12)$ \\
$\AK$ & $\tfrac{1}{3}(n - 1)(2n^2 + 47n - 60)$ \\
$\Nork$ & $2(n - 1)(n^2 + 8n - 10)$ \\
$\Kook$ & $2(n - 1)(n^2 + 4n - 2)$ \\
\hline
\end{tabular}
\end{center}
The corresponding strengths are obtained by dividing each $T_P(n)$ by the 
normalization factor $n^2(n^2 - 1)$.
The double bishop has $T_{\B\B}(n) = 2 T_\B(n)$ and $\str(\B\B) = 2 \str(\B)$.
\end{theorem}

\begin{proof} 
The denominator $n^2(n^2 - 1)$ in the definition of strength factors as $n^2 (n - 1)(n + 1)$.
We compute each total $T_P(n)$ explicitly and divide.

\emph{Rook.} 
For any square, the \R{} reaches exactly $2(n - 1)$ other squares on the same rank or file.
Summing, $T_\R(n) = 2(n - 1) \cdot n^2$, giving $\str(\R) = 2/(n + 1)$.

\emph{King.} 
The four corner squares each have three \K-moves; the $4(n - 2)$ non-corner perimeter squares each have five; the remaining $n^2 - 4(n - 1)$ interior squares each have eight.
Hence, $T_\K(n) = 4 \cdot 3 + 4(n - 2) \cdot 5 + (n^2 - 4(n - 1)) \cdot 8 = 4(n - 1)(2n - 1)$, and the formula follows.

\emph{Knight.} 
For each square $(i, j) \in [n]^2$, the number of legal \N-moves is the number of offsets $(\pm 1, \pm 2)$ and $(\pm 2, \pm 1)$ that remain on the board.
The eight knight offsets fall into four pairs, and the count of legal moves at $(i, j)$ depends only on whether $i$ and $j$ are at distance at least $2$ from the boundary in the row and column directions.
A direct case analysis gives the per-square mobility, broken down by shell.
In $B_1$ (the perimeter), the four corners contribute $2$ moves each, the eight squares adjacent to corners along the perimeter contribute $3$ each, and the remaining $4(n - 4)$ perimeter squares contribute $4$ each.
In $B_2$ (the second ring), the four ``next-to-corner'' squares (diagonally adjacent to the corners) contribute $4$ moves each, and the remaining $4(n - 4)$ squares of $B_2$ contribute $6$ each.
All remaining interior squares, of which there are $(n - 4)^2$, contribute $8$ moves each.
Summing,
\begin{equation*}
T_\N(n) \,=\, 4 \cdot 2 + 8 \cdot 3 + 4(n - 4) \cdot 4
+ 4 \cdot 4 + 4(n - 4) \cdot 6 + (n - 4)^2 \cdot 8
\,=\, 8(n - 1)(n - 2),
\end{equation*}
valid for $n \geq 4$.
The formula $8(n - 1)(n - 2)$ also gives the correct totals on the small boards $n = 2$ and $n = 3$, namely $T_\N(2) = 0$ and $T_\N(3) = 16$, by direct enumeration; in particular, the formula extends to all $n \geq 2$.

\emph{Double bishop.} 
For each square in shell $B_i$, the number of diagonal moves available to a $\B\B$ is $n - 1 + 2(i - 1)$.
Indeed, a diagonal move from $B_i$ runs ``inward'' to a maximum of $i - 1$ steps in each of the two inward diagonal directions and ``outward'' to the edge in the other two; the total is independent of the specific square within $B_i$.
Summing over $i$,
\begin{align*}
T_{\B\B}(n) &= \sum_{i = 1}^{\lfloor n/2 \rfloor} 4(n - 2i + 1)
(n - 1 + 2(i - 1)) \\
&= 4 \sum_{i = 1}^{\lfloor n/2 \rfloor} (n + 1 - 2i)(n - 3 + 2i) \\
&= \tfrac{2}{3} n(n - 1)(2n - 1).
\end{align*}
The last equality follows from expanding and applying the elementary identities $\sum_{i = 1}^{k} i = k(k+1)/2$ and $\sum_{i = 1}^{k} i^2 = k(k+1)(2k+1)/6$, with $k = n/2$ when $n$ is even.
When $n$ is odd, the sum $\sum_{i=1}^{\lfloor n/2 \rfloor}$ covers all shells except the single central square, which lies in a shell of its own and contributes $2(n-1)$ diagonal moves ($(n-1)/2$ steps in each of four diagonal directions).
Adding this to the shell sum and simplifying gives the same closed form $\tfrac{2}{3}n(n-1)(2n-1)$, so the formula holds for all $n \geq 1$.
Dividing by 2 gives $T_\B(n) = \tfrac{1}{3} n(n - 1)(2n - 1)$.

\emph{Compounds.}
 The nine compound pieces in $\Pieces$ are by definition the unions of disjoint basic-piece move sets, so each total is a sum.
For the five standard fairy-chess compounds,
\begin{align*}
T_\Q(n) &= T_{\B\B}(n) + T_\R(n) = \tfrac{2}{3} n(n - 1)(5n - 1), \\
 T_\AB(n) &= T_{\B\B}(n) + T_\N(n) = \tfrac{2}{3}(n - 1)(2n^2 + 11n - 24), \\
T_\Em(n) &= T_\R(n) + T_\N(n) = 2(n - 1)(n^2 + 4n - 8), \\
T_\C(n) &= T_\K(n) + T_\N(n) = 4(n - 1)(4n - 5), \\
T_\Am(n) &= T_\Q(n) + T_\N(n) = \tfrac{2}{3}(n - 1)(5n^2 + 11n - 24).
\end{align*}
For the four further compounds,  the same procedure applies:
\begin{align*}
T_\Bee(n) &= T_\K(n) + T_\B(n) = \tfrac{1}{3}(n - 1)(2n - 1)(n + 12), \\
T_\AK(n) &= T_\K(n) + T_\N(n) + T_\B(n) = \tfrac{1}{3}(n - 1)(2n^2 + 47n - 60), \\
T_\Nork(n) &= T_\K(n) + T_\N(n) + T_\R(n) = 2(n - 1)(n^2 + 8n - 10), \\
T_\Kook(n) &= T_\K(n) + T_\R(n) = 2(n - 1)(n^2 + 4n - 2).
\end{align*}
Dividing each $T_P(n)$ by $n^2(n^2 - 1) = n^2(n - 1)(n + 1)$ yields the stated strengths.
\end{proof}

We shall  also have occasion to refer to four further compound pieces.
The \Bee{}  is equal to $\B + \K$, with $T_\Bee = T_\K + T_\B$.
The  \AK{}  is $\B + \N + \K$, the \Nork{}  is $\K + \N + \R$, and the \Kook{}  is $\K + \R$.
We collect the thirteen pieces of interest in the set
\begin{multline}
\label{eq:S-def}
\Pieces \,=\, \{\K, \N, \B, \R, \Q, \Am, \Em, \C, \AB, \\
\AK, \Nork, \Kook, \Bee\}.
\end{multline}

\begin{remark} 
The list $\Pieces$ exhausts the distinct compound pieces that can be formed from the basic set $\{\K, \N, \B, \R\}$ by union of move sets, modulo two reductions.
First, since $\Q = \B\B + \R$ and $\Am = \Q + \N$ already incorporate the king's moves implicitly, $\K + \B + \R$ duplicates the \Q{} and $\K + \B + \R + \N$ duplicates the \Am{}.
Second, the list distinguishes the (single-color) \B{} from the double bishop $\B\B$, the latter not being among the thirteen members of $\Pieces$ but appearing implicitly inside the \Q, the \Am, and the \AB.
The reader will encounter the same thirteen pieces in our discussion of armies, where they form the alphabet of allowed pieces.
\end{remark}

\begin{remark} \label{rem:ehrhart} 
The move totals $T_P(n)$ admit a natural interpretation in the language of polytope-counting.
For a fixed move set $V_P \subset \mathbb{Z}^2$, the total $T_P(n)$ counts ordered pairs of distinct squares $(p, q) \in [n]^2 \times [n]^2$ such that $q - p \in V_P$, or equivalently, integer points in the relative interior of the dilation of a piece-determined polytope inscribed in $[0, n - 1]^2 \times [0, n - 1]^2$.
For each basic piece, $T_P(n)$ is a polynomial in $n$, of degree $2$ for the leapers ($\K$ and $\N$) and degree $3$ for the riders ($\B$ and $\R$); the parity decomposition we develop in Section~\ref{sec:parity} shows that the parity-component totals $T_P^\pm(n)$ are quasi-polynomial of period at most $2$, illustrating the standard Ehrhart-style behavior for counts in dilated rational polytopes.
The $q$-Queens series of Chaiken, Hanusa, and Zaslavsky~\cite{ChaikenHanusaZaslavsky} develops the systematic Ehrhart-quasi-polynomial framework for non-attacking chess-piece placement; the present paper studies the simpler related quantity $T_P(n)$ and its arithmetic.
\end{remark} 

\begin{definition}[Nontrivial Board]
\label{def:nontrivial-board}
A board size $n$ is \emph{nontrivial} if every basic piece 
 has at least one legal move from every square; equivalently, $n \geq 4$.
\end{definition}

Indeed, recall that the basic pieces are those in $\{\K, \N, \B, \R\}$;
 for $n = 3$ the \N{} has no legal moves from b2, and
for $n \leq 2$ the \N{} has no legal moves at all.
We restrict attention to nontrivial boards throughout.

\section{Asymptotic Structure} \label{sec:asymptotics}

Before turning to the formal asymptotic analysis, it is useful to visualize the strengths of the principal pieces as functions of $n$.
Figure~\ref{fig:strengths} plots the strengths of the five standard chess pieces and the four standard fairy-chess compounds on a log-log scale.
Several of the paper's main phenomena are visible at once.
The two visual slopes correspond to the two asymptotic decay rates ($1/n$ for riders, $1/n^2$ for leapers); the magic-board coincidences appear as crossings at $n \in \{6, 8, 12\}$; the \B{}-versus-\N{} transition occurs near $n \approx 11$; and beyond $n = 24$ no further crossings take place.

\begin{figure}[h]
\centering
\begin{tikzpicture}
\begin{axis}[
  width=14cm, height=9.5cm,
  xmode=log, log basis x=10,
  xmin=4, xmax=30,
  ymode=log, log basis y=10,
  ymin=0.005, ymax=1.05,
  xlabel={$n$},
  ylabel={$\str(P)$},
  xtick={4,5,6,7,8,9,10,12,14,16,18,20,24,28},
  log ticks with fixed point,
   grid=major,
  grid style={gray!15},
  no markers,
]
\addplot[solid, line width=1.1pt, black] coordinates {(4,0.83333)(5,0.69333)(6,0.58730)(7,0.50680)(8,0.44444)(9,0.39506)(10,0.35515)(11,0.32231)(12,0.29487)(13,0.27163)(14,0.25170)(15,0.23444)(16,0.21936)(17,0.20607)(18,0.19428)(19,0.18375)(20,0.17429)(21,0.16574)(22,0.15798)(23,0.15091)(24,0.14444)(25,0.13850)};
\node[anchor=west, font=\small] at (axis cs:25.5,0.155) {Am};

\addplot[solid, line width=1.0pt, black!80] coordinates {(4,0.63333)(5,0.53333)(6,0.46032)(7,0.40476)(8,0.36111)(9,0.32593)(10,0.29697)(11,0.27273)(12,0.25214)(13,0.23443)(14,0.21905)(15,0.20556)(16,0.19363)(17,0.18301)(18,0.17349)(19,0.16491)(20,0.15714)(21,0.15007)(22,0.14361)(23,0.13768)(24,0.13222)(25,0.12718)};
\node[anchor=west, font=\small] at (axis cs:25.5,0.118) {Q};

\addplot[solid, line width=0.95pt, black!65] coordinates {(4,0.60000)(5,0.49333)(6,0.41270)(7,0.35204)(8,0.30556)(9,0.26914)(10,0.24000)(11,0.21625)(12,0.19658)(13,0.18005)(14,0.16599)(15,0.15389)(16,0.14338)(17,0.13418)(18,0.12606)(19,0.11884)(20,0.11238)(21,0.10658)(22,0.10133)(23,0.09657)(24,0.09222)(25,0.08825)};
\node[anchor=west, font=\small] at (axis cs:25.5,0.092) {Em};

\addplot[solid, line width=0.95pt, black!50] coordinates {(4,0.40000)(5,0.33333)(6,0.28571)(7,0.25000)(8,0.22222)(9,0.20000)(10,0.18182)(11,0.16667)(12,0.15385)(13,0.14286)(14,0.13333)(15,0.12500)(16,0.11765)(17,0.11111)(18,0.10526)(19,0.10000)(20,0.09524)(21,0.09091)(22,0.08696)(23,0.08333)(24,0.08000)(25,0.07692)};
\node[anchor=west, font=\small] at (axis cs:25.5,0.077) {R};

\addplot[solid, line width=0.95pt, black!40] coordinates {(4,0.43333)(5,0.36000)(6,0.30159)(7,0.25680)(8,0.22222)(9,0.19506)(10,0.17333)(11,0.15565)(12,0.14103)(13,0.12877)(14,0.11837)(15,0.10944)(16,0.10172)(17,0.09496)(18,0.08902)(19,0.08375)(20,0.07905)(21,0.07483)(22,0.07103)(23,0.06758)(24,0.06444)(25,0.06158)};
\node[anchor=west, font=\small] at (axis cs:25.5,0.060) {AB};

\addplot[solid, line width=0.9pt, black!30] coordinates {(4,0.11667)(5,0.10000)(6,0.08730)(7,0.07738)(8,0.06944)(9,0.06296)(10,0.05758)(11,0.05303)(12,0.04915)(13,0.04579)(14,0.04286)(15,0.04028)(16,0.03799)(17,0.03595)(18,0.03411)(19,0.03246)(20,0.03095)(21,0.02958)(22,0.02833)(23,0.02717)(24,0.02611)(25,0.02513)};
\node[anchor=west, font=\small] at (axis cs:25.5,0.0285) {B};

\addplot[dashed, line width=1.0pt, black] coordinates {(4,0.55000)(5,0.40000)(6,0.30159)(7,0.23469)(8,0.18750)(9,0.15309)(10,0.12727)(11,0.10744)(12,0.09188)(13,0.07946)(14,0.06939)(15,0.06111)(16,0.05423)(17,0.04844)(18,0.04353)(19,0.03934)(20,0.03571)(21,0.03257)(22,0.02982)(23,0.02741)(24,0.02528)(25,0.02338)};
\node[anchor=west, font=\small] at (axis cs:25.5,0.0210) {C};

\addplot[dashed, line width=0.95pt, black!65] coordinates {(4,0.35000)(5,0.24000)(6,0.17460)(7,0.13265)(8,0.10417)(9,0.08395)(10,0.06909)(11,0.05785)(12,0.04915)(13,0.04227)(14,0.03673)(15,0.03222)(16,0.02849)(17,0.02537)(18,0.02274)(19,0.02050)(20,0.01857)(21,0.01690)(22,0.01545)(23,0.01418)(24,0.01306)(25,0.01206)};
\node[anchor=west, font=\small] at (axis cs:25.5,0.0135) {K};

\addplot[dashed, line width=0.9pt, black!40] coordinates {(4,0.20000)(5,0.16000)(6,0.12698)(7,0.10204)(8,0.08333)(9,0.06914)(10,0.05818)(11,0.04959)(12,0.04274)(13,0.03719)(14,0.03265)(15,0.02889)(16,0.02574)(17,0.02307)(18,0.02079)(19,0.01884)(20,0.01714)(21,0.01567)(22,0.01437)(23,0.01323)(24,0.01222)(25,0.01132)};
\node[anchor=west, font=\small] at (axis cs:25.5,0.0105) {N};

\draw[gray, dotted, thick] (axis cs:6,0.005) -- (axis cs:6,1.05);
\draw[gray, dotted, thick] (axis cs:8,0.005) -- (axis cs:8,1.05);
\draw[gray, dotted, thick] (axis cs:12,0.005) -- (axis cs:12,1.05);
\draw[gray, dashed, thick] (axis cs:24,0.005) -- (axis cs:24,1.05);

\node[circle, draw=black, fill=white, inner sep=1.3pt] at (axis cs:8,0.22222) {};
\node[circle, draw=black, fill=white, inner sep=1.3pt] at (axis cs:6,0.30159) {};
\node[circle, draw=black, fill=white, inner sep=1.3pt] at (axis cs:12,0.04915) {};

\end{axis}
\end{tikzpicture}
\caption{Strengths of the nine pieces of standard and fairy-chess
interest as functions of the board size $n$, plotted on a log-log scale. 
Solid lines indicate riders; dashed lines indicate leapers.
Within each group, the gray level varies from darkest (strongest) to
lightest (weakest) to aid identification. 
The two visual slopes correspond to the two asymptotic classes of Theorem~\ref{thm:asymptotics}. 
The dotted vertical lines mark the three magic boards $n = 6, 8, 12$, with open circles at the
corresponding strength coincidences (Theorem~\ref{thm:coincidences}).
The dashed vertical line at $n = 24$ marks the threshold of the
stable ordering (Theorem~\ref{thm:stable-ordering}).}
\label{fig:strengths}
\end{figure}
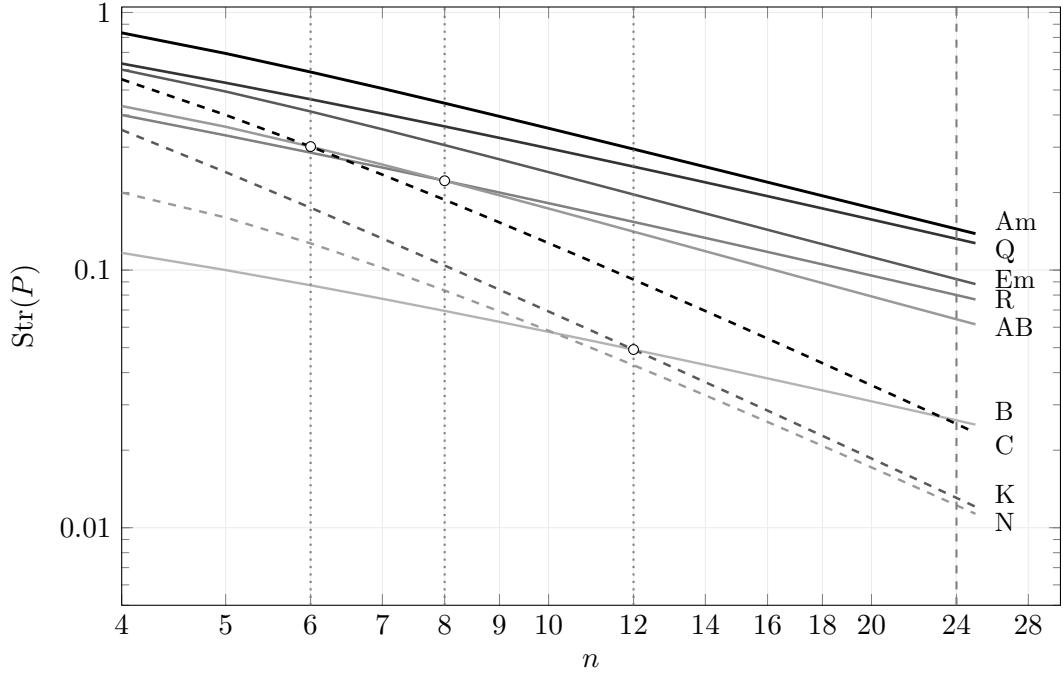

\subsection{The Rider/Leaper Dichotomy}

The strengths in Theorem~\ref{thm:strengths} all tend to zero as $n \to \infty$, but in fact at two distinct rates.

\begin{definition} \label{def:rider-leaper} A chess piece is a \emph{rider} if its move set includes at least one direction in which it can move arbitrarily far, until blocked by the edge of the board.
It is a \emph{leaper} if it is not a rider; that is, if its move set consists of finitely many fixed offsets.
\end{definition}

Among the thirteen pieces of $\Pieces$ that we introduced previously, the riders are the ten pieces $\R$, $\B$, $\Q$, $\AB$, $\Em$, $\Am$, $\Bee$, $\AK$, $\Nork$, $\Kook$, while the leapers are the three pieces $\K$, $\N$, $\C$.
The double bishop $\B\B$, although not a member of $\Pieces$, is also a rider.

\begin{theorem} \label{thm:asymptotics} For each piece $P \in \Pieces$, there exist constants $c_P > 0$ and $\alpha_P \in \{1, 2\}$  
such that 
$$
 \str(P) = \frac{c_P}{n^{\alpha_P}} + O\!\left(\frac{1}{n^{ \alpha_P  + 1 } }\right) \qquad\text{as } n \to \infty. 
 $$ 
The exponents and leading constants are given by the table below; in particular, $\alpha_P = 1$ if and only if $P$ is a rider, and $\alpha_P = 2$ if and only if $P$ is a leaper.

\begin{center} \renewcommand{\arraystretch}{1.2}
\begin{tabular}{| l | c | c |}
\hline
Piece & $\alpha_P$ & $c_P$ \\ \hline
$\K, \N$ & $2$ & $8$ \\
$\C$ & $2$ & $16$ \\ \hline
$\B, \Bee, \AK$ & $1$ & $2/3$ \\
$\AB$ & $1$ & $4/3$ \\
$\R, \Em, \Nork, \Kook$ & $1$ & $2$ \\
$\Q, \Am$ & $1$ & $10/3$ \\ \hline
\end{tabular}
\end{center}
The double bishop $\B\B$, although not a member of $\Pieces$, satisfies $\alpha_{\B\B} = 1$ and $c_{\B\B} = 4/3$ by the same analysis.
\end{theorem}

\begin{proof} 
 Each strength is a rational function in $n$ of the form $\str(P) = N_P(n) / D(n)$ in which $D(n) = n^2(n + 1)$ or $D(n) = n(n + 1)$ and $N_P$ is a polynomial whose degree is $\deg D - \alpha_P$.
Expanding in $1/n$ and reading off the leading term gives the table.

For example,   $\str(\K) = 4(2n - 1)/(n^2(n + 1))$.
Writing the numerator as $8n - 4$ and the denominator as $n^3 + n^2$, 
we have 
 $$ 
  \str(\K) = \frac{8n - 4}{n^3 + n^2} = \frac{8}{n^2} \cdot \frac{1 - 1/(2n)}{1 + 1/n} = \frac{8}{n^2} - \frac{12}{n^3} + O(1/n^4), 
  $$ 
 so $\alpha_\K = 2$ and $c_\K = 8$.

For $\str(\R) = 2/(n + 1)$, the expansion is even simpler: 
$$ 
 \str(\R) \,=\, \frac{2}{n} \cdot \frac{1}{1 + 1/n} \,=\, \frac{2}{n} - \frac{2}{n^2} + O(1/n^3), 
 $$ 
so $\alpha_\R = 1$ and $c_\R = 2$.

The remaining cases  follow by direct expansion of the formulas in Theorem~\ref{thm:strengths}.
The compound formulas combine: for any $P_1, P_2$ whose move sets are disjoint, $\str(P_1 + P_2) = \str(P_1) + \str(P_2)$, so $\alpha_{P_1 + P_2} = \min(\alpha_{P_1}, \alpha_{P_2})$, and $c_{P_1 + P_2} = c_{P_1} + c_{P_2}$ if the exponents are equal, else $c_{P_1 + P_2} = c_{P_i}$ for the dominant component.
This explains why $\Bee = \K + \B$ inherits $c = 2/3$ from the \B{} (the \K{}'s contribution, of order $1/n^2$, is dominated), and furthermore why $\AB = \B\B + \N$ inherits $c = 4/3$ from the $\B\B$, and so on.
\end{proof}

\begin{remark}
  The rider/leaper dichotomy admits an underlying geometric explanation.
Each rider's set of legal arrows in the unit-square embedding of the board sweeps out a $3$-dimensional subset of the $4$-dimensional arrow space; each leaper's sweeps out only a $2$-dimensional subset.
The factor of $1/n$ difference in their strengths is exactly the codimension difference.
We aim to develop this explanation rigorously in a separate paper.
\end{remark}

\subsection{Asymptotic Strength Ratios}

The four leading constants for riders ($2/3, 4/3, 2, 10/3$)  
 all share $c_\B = 2/3$ as a divisor.

\begin{corollary} \label{cor:ratios} 
For any rider $P \in \Pieces$ it holds that 
$\lim_{n \to \infty} \str(P) / \str(\B) \in \{1, 2, 3, 5\}$.
Specifically, the following hold.
\begin{itemize}[itemsep=2pt, topsep=4pt]
 \item  Limit $1$: $\B, \Bee, \AK$.
\item Limit $2$: $\AB$.
\item Limit $3$: $\R, \Em, \Nork, \Kook$.
\item Limit $5$: $\Q, \Am$.
\end{itemize}
The double bishop $\B\B$, although not a member of $\Pieces$,  by the same calculation from above 
satisfies
$\lim_{n \to \infty} \str(\B\B) / \str(\B) = 2$.
\end{corollary}

\begin{proof} This follows directly from Theorem~\ref{thm:asymptotics}, 
since $\str(P)/\str(\B)$ converges to $c_P/c_\B$, which equals $(3/2) c_P$ in each case, 
and $(3/2) \cdot \{2/3, 4/3, 2, 10/3\} = \{1, 2, 3, 5\}$.
\end{proof}

For the leapers, the analogous result is even simpler.
We have $\str(\K)/\str(\N) \to 1$ and $\str(\C)/\str(\K) \to 2$.
The fact that $\K$ and $\N$ have identical leading constants (both $c = 8$) is particularly notable and is responsible for several finer phenomena discussed below.

\subsection{Stable Ordering}

For small $n$, the strength order of the thirteen pieces of $\Pieces$ fluctuates.
As we shall see, the \N{} is stronger than the \B{} for $n \leq 10$, the \R{} equals the \AB{} at $n = 8$, and so on.
From some threshold $n^*$ onward, the order must, however, be fixed by asymptotic considerations.
Our next theorem below determines $n^*$ exactly.

\begin{lemma} \label{lem:king-knight-positive} 
For every $n \geq 2$,
\begin{equation*}
\str(\K) - \str(\N) \,=\, \frac{12}{n^2(n + 1)} \,>\, 0.
\end{equation*}
\end{lemma}

\begin{proof} 
By Theorem~\ref{thm:strengths}, $\str(\K)$ and $\str(\N)$ 
share the denominator $n^2(n + 1)$, with numerators $4(2n - 1)$ and $8(n - 2)$ respectively.
The difference of numerators is $4(2n - 1) - 8(n - 2) = 12$, independent of $n$.
Positivity follows.
\end{proof}

\begin{theorem}[Stable Ordering Theorem]
\label{thm:stable-ordering} 
For all $n \geq 24$, the thirteen pieces of $\Pieces$ satisfy the strict strength ordering
\begin{align*}
\str(\Am) > \str(\Q) &> \str(\Nork) > \str(\Kook) > \str(\Em) > \str(\R) \\
&> \str(\AB) > \str(\AK) > \str(\Bee) > \str(\B) \\
&> \str(\C) > \str(\K) > \str(\N).
\end{align*}
The threshold $n = 24$ is sharp; specifically, the inequality $\str(\B) > \str(\C)$ fails at $n = 23$.
\end{theorem}

\begin{proof} 
There are twelve adjacent inequalities to verify; we group them by asymptotic class.

\medskip

\textit{Within each asymptotic class.} 
Each within-class inequality reduces, after subtracting the two strengths, to a simple expression that is positive for $n \geq 3$.

For the rider class with $c_P = 10/3$ (namely $\Am, \Q$),
\begin{equation*}
\str(\Am) - \str(\Q) \,=\, \str(\N) \,>\, 0
\qquad (n \geq 3).
\end{equation*}

For the rider class with $c_P = 2$, namely $\Nork, \Kook, \Em, \R$, the three adjacent identities are
\begin{align*}
\str(\Nork) - \str(\Kook) &\,=\, \str(\N), \\
\str(\Kook) - \str(\Em) &\,=\, \str(\K) - \str(\N), \\
\str(\Em) - \str(\R) &\,=\, \str(\N).
\end{align*}
Each is positive for $n \geq 3$,  with positivity of the middle identity relying on the King-Knight identity $\str(\K) - \str(\N) = 12/(n^2(n + 1)) > 0$ from Lemma~\ref{lem:king-knight-positive}.

For the rider class with $c_P = 2/3$, namely $\AK, \Bee, \B$,
\begin{align*}
\str(\AK) - \str(\Bee) &\,=\, \str(\N), \\
\str(\Bee) - \str(\B) &\,=\, \str(\K),
\end{align*}
both positive.

For the leaper class $\C, \K, \N$,
\begin{align*}
\str(\C) - \str(\K) &\,=\, \str(\N), \\
\str(\K) - \str(\N) &\,>\, 0,
\end{align*}
the second by Lemma~\ref{lem:king-knight-positive}.

\medskip

\textit{Across asymptotic classes.} 
Four boundary inequalities remain: between the $c = 10/3$ class and the $c = 2$ class; between $c = 2$ and $c = 4/3$; between $c = 4/3$ and $c = 2/3$; and between $c = 2/3$ (riders) and the leapers.

\noindent (i) The inequality $\str(\Q) > \str(\Nork)$ is equivalent to 
$T_\Q - T_\Nork > 0$, that is, to
\begin{equation*}
\frac{2(n - 1)(2n^2 - 25n + 30)}{3} > 0.
\end{equation*}
The roots of $2n^2 - 25n + 30$ are $(25 \pm \sqrt{385})/4 \approx 1.34, 11.16$, 
so the quadratic is positive for integers $n \geq 12$.

\noindent (ii) The inequality $\str(\R) > \str(\AB)$ is equivalent to $T_\R - T_\AB > 0$.
We compute
\begin{equation*}
T_\R - T_\AB \,=\, \frac{2(n - 1)(n - 3)(n - 8)}{3},
\end{equation*}
which has the same sign as $(n - 3)(n - 8)$ on integers $n \geq 4$, and is positive for $n \geq 9$.

\noindent (iii) For $\str(\AB) > \str(\AK)$, 
we rewrite using the atomic decomposition,
\begin{align*}
T_\AB - T_\AK
&\,=\, (T_{\B\B} + T_\N) - (T_\K + T_\N + T_\B) \,=\, T_\B - T_\K \\
&\,=\, \frac{(n - 1)(2n - 1)(n - 12)}{3}.
\end{align*}
This is positive for integers $n \geq 13$.

\noindent (iv) The inequality $\str(\B) > \str(\C)$ is evidently equivalent to $T_\B - T_\C > 0$.
Computing,
\begin{equation*}
T_\B - T_\C \,=\, \tfrac{1}{3} n(n - 1)(2n - 1) - 4(n - 1)(4n - 5)
\,=\, \frac{(n - 1)(2n^2 - 49n + 60)}{3}.
\end{equation*}
The roots of $2n^2 - 49n + 60$ are $(49 \pm \sqrt{1921})/4 \approx 1.29, 23.21$, 
so the quadratic is positive for integers $n \geq 24$.

\medskip

The full ordering therefore holds for $n \geq \max\{3, 9, 12, 13, 24\} = 24$.
The binding constraint is (iv).
At $n = 23$, direct computation gives $T_\B(23) = 7590$ and $T_\C(23) = 7656$, so $\str(\B) < \str(\C)$ at $n = 23$.
Hence, $n^* = 24$ is sharp.
\end{proof}

\begin{remark} 
The threshold $n^* = 24 = 2 \cdot 12$ is suggestive.
The single \B{} overtakes the \K{} at $n = 12$ (Proposition~\ref{thm:king-bishop} below), but overtaking the \C{} (which is essentially $\K + \N$, with $\K \approx \N$ asymptotically) takes roughly twice as long.
The exact threshold $24$ reflects the lower-order corrections.
\end{remark}

\section{The Classification of Strength Coincidences} \label{sec:coincidences}

In this section, we turn to the central structural theorem of the paper.

\subsection{The Archbishop-Rook Coincidence}

We have already noted, in Remark~\ref{rem:rook-archbishop}, that $\str(\R) = \str(\AB)$ on the $8 \times 8$ board.
We now show that this identity, viewed as an equation in $n$, has only $n = 8$ as a nontrivial solution.

\begin{proposition} \label{thm:archbishop-rook} 
The $8 \times 8$ chessboard is the unique nontrivial board for which we have the equality $\str(\R) = \str(\AB)$.
\end{proposition}

\begin{proof} 
By Theorem~\ref{thm:strengths}, the equation $\str(\R) = \str(\AB)$ is equivalent to
\begin{equation*}
2(n - 1) n^2 \,=\, \tfrac{2}{3}(n - 1)(2n^2 + 11n - 24).
\end{equation*}
Canceling the factor $(n - 1)$ (nonzero for $n \geq 2$) and clearing fractions yields $3n^2 = 2n^2 + 11n - 24$, that is, $n^2 - 11n + 24 = 0$.
This factors as $(n - 3)(n - 8) = 0$.
Of the two integer solutions, $n = 3$ is excluded by nontriviality; $n = 8$ remains.
\end{proof}

\subsection{The Knight Versus Bishop Comparison} \label{sec:knight-bishop}

Another source of arithmetic identities comes from the comparison of the \K{} with the \B{}.
The following two propositions are immediate from the closed-form expressions.

\begin{proposition} \label{thm:king-double-bishop} 
The $6 \times 6$ chessboard is the unique nontrivial
   board for which we have the equality $\str(\K) = \str(\B\B)$ (equivalently, $T_\K = 2 T_\B$).
\end{proposition}

\begin{proof} 
The equation $T_\K = 2 T_\B$ becomes
\begin{equation*}
4(n - 1)(2n - 1) \,=\, \tfrac{2}{3} n(n - 1)(2n - 1).
\end{equation*}
Canceling  $(n - 1)(2n - 1)$, which is nonzero factor for $n \geq 2$, 
reduces this to $12 = 2n$, 
i.e., $n = 6$.
\end{proof}

\begin{proposition} \label{thm:king-bishop} 
The $12 \times 12$ chessboard is the unique nontrivial board for 
 which we have the equality $\str(\K) = \str(\B)$.
\end{proposition}

\begin{proof} 
The same cancellation as in Proposition~\ref{thm:king-double-bishop}
   reduces $T_\K = T_\B$ to $n = 12$.
\end{proof}

The following result is, by contrast, an observation of impossibility, and was 
in fact  the original motivation for our investigation of coincidences.

\begin{proposition} \label{prop:knight-bishop-no-coincidence} 
There is no integer $n \geq 1$ with $\str(\N) = \str(\B)$.
\end{proposition}

\begin{proof} 
The equation $\str(\N) = \str(\B)$ becomes
\begin{equation*}
\frac{8(n - 2)}{n^2(n + 1)} \,=\, \frac{1}{3} \cdot
\frac{2n - 1}{n(n + 1)},
\end{equation*}
which simplifies to $24(n - 2) = n(2n - 1)$, that is, $2n^2 - 25n + 48 = 0$.
The discriminant is $625 - 384 = 241$, which is not a perfect square since $15^2 = 225 < 241 < 256 = 16^2$.
The roots are therefore irrational, hence in particular not integers.
\end{proof}

The two real roots of $2n^2 - 25n + 48$ are $n_\pm = (25 \pm \sqrt{241})/4 \approx 2.369,\ 10.131$.
The quadratic is negative on the interval $(n_-, n_+)$ and positive outside it, which translates, via the equivalence in the proof above, into the following dichotomy.
We have $\str(\N) > \str(\B)$ for integers $4 \leq n \leq 10$, and $\str(\B) > \str(\N)$ for integers $n \geq 11$.
The transition occurs between $n = 10$ and $n = 11$, 
and $n = 10$ is the closest the two strengths 
ever come on a nontrivial integer board.

\begin{theorem} \label{thm:bishop-knight-min} 
Among nontrivial boards $n \geq 4$, the value $n = 10$ uniquely minimizes the gap $\str(\N) - \str(\B) > 0$, with 
$$
 \str(\N) - \str(\B) \big|_{n = 10} = \frac{1}{1650} \approx 0.0606\%. 
$$ 
The integer $n = 11$ is the smallest at which the \B{} becomes stronger than the \N{}, and after $n = 11$ the gap $\str(\B) - \str(\N)$ remains positive but tends back to zero as $n \to \infty$.
\end{theorem}

\begin{proof} 
The difference is 
$$
 \str(\B) - \str(\N) = \frac{2n^2 - 25n + 48}{3 n^2 (n + 1)}, 
 $$ 
which is the quotient of the quadratic $q(n) = 2n^2 - 25n + 48$ discussed above by a positive denominator.
Therefore, $\str(\N) - \str(\B) > 0$ exactly on integers $n$ with $q(n) < 0$, i.e., 
$4 \leq n \leq 10$; 
on this range, the closest that $q(n)$ comes to $0$ from below is at 
the integer nearest the upper root $n_+ \approx 10.131$, namely $n = 10$, where $q(10) = -2$ and the difference is $-2/(3 \cdot 100 \cdot 11) = -1/1650$.

For $n \geq 11$, we have that both $\str(\B)$ and $\str(\N)$ tend to $0$ as $n \to \infty$ (with $\str(\B) = \Theta(1/n)$ and $\str(\N) = \Theta(1/n^2)$, by Theorem~\ref{thm:asymptotics}), so their difference also tends to $0$, but with $\str(\B)$ dominant for all $n \geq 11$.
\end{proof}

\begin{remark} 
Theorem~\ref{thm:bishop-knight-min} provides, within our model, a small mathematical observation favoring the $10 \times 10$ board on ``fairness'' grounds.
The $10 \times 10$ board appears in several chess variants featuring the \AB{} and \Em{}, most notably the Grand Chess of Freeling~\cite{Pritchard2007}, and was among the board sizes Capablanca experimented with before settling on the $10 \times 8$ board for Capablanca's chess; our result is one possible quantitative argument in favor of the $10 \times 10$ size.
\end{remark}

\subsection{The King-Knight Identity} \label{sec:king-knight}

Whereas the \B{} and the \N{},   being pieces of different asymptotic classes, 
exhibit a transition in their strength ordering at $n = 11$ and a near-coincidence at $n = 10$, 
the situation between the \K{} and the \N{}, both of which are leapers with leading constant $c = 8$, is qualitatively different.
Indeed, although Theorem~\ref{thm:asymptotics} only tells us that $\str(\K)/\str(\N) \to 1$ as $n \to \infty$,
   the exact difference between the two strengths admits a closed form that is striking in its simplicity.

\begin{theorem} \label{thm:king-knight} 
 For every integer $n \geq 2$,
\begin{equation}
\label{eq:king-knight}
\str(\K) - \str(\N) \,=\, \frac{12}{n^2(n + 1)}.
\end{equation}
In particular, $\str(\K) > \str(\N)$ for all $n \geq 2$, and the gap decreases monotonically with $n$.
\end{theorem}

\begin{proof} 
The closed-form identity is Lemma~\ref{lem:king-knight-positive}; 
equivalently, at the level of move totals, 
it holds that  
$T_\K(n) - T_\N(n) = 12(n - 1)$.
The monotonicity statement is immediate from~\eqref{eq:king-knight}, 
since the right-hand side is a strictly decreasing positive function of $n \geq 2$.
\end{proof}

Note that the identity~\eqref{eq:king-knight} is exact, not asymptotic, and holds for 
 every integer $n \geq 2$.
The numerator constant $12$ is the unique nontrivial structural invariant 
 attached to the \K{}-versus-\N{} comparison.
Several remarks are in order.

\begin{remark} \label{rem:king-knight-uniform} 
The \K{} dominates the \N{} uniformly in $n$, 
 in sharp contrast to the \B{}-versus-\N{} comparison studied in Section~\ref{sec:knight-bishop}.
The latter exhibits a transition at $n = 11$, whereas the \K{}-versus-\N{} ordering is fixed at every nontrivial board.
Consequently, the \K{} is, in our model,  the strongest leaper basic piece on every $n \times n$ board.
\end{remark}

\begin{remark}
\label{rem:kook-empress}
An analogous identity follows from Theorem~\ref{thm:king-knight} by adding the \R{} to both pieces.
Indeed, since $\Kook = \K + \R$ and $\Em = \N + \R$, 
it also holds that  $T_\Kook - T_\Em = T_\K - T_\N$, and therefore,
\begin{equation*}
\str(\Kook) - \str(\Em) \,=\, \frac{12}{n^2(n + 1)}.
\end{equation*}
These are the only two pairs $(P, Q)$ of distinct pieces in $\Pieces$ for which $T_P - T_Q$ is a constant multiple of $(n - 1)$, 
and the second is recoverable from the first by the substitution $P \mapsto P + \R$.
In this sense Theorem~\ref{thm:king-knight} captures the only independent strength difference of the form $c/(n^2(n+1))$ 
with $c$ a nonzero rational constant, among pairs of pieces in $\Pieces$.
\end{remark}

\subsection{The Bishop-King Proportionality} \label{sec:bishop-king-proportionality}

  The previous two subsections compared the \K{} and \N{}, and the \B{} and \N{}, in turn.
We now turn to the comparison between the \K{} and the \B{}, which in fact exhibits an even tighter relationship.
Specifically, the move totals $T_\K(n)$ and $T_\B(n)$ are not merely related by a near-coincidence or by a constant difference; 
they are, as functions of $n$, exactly proportional, with proportionality constant a linear function of $n$.

\begin{theorem} \label{thm:bishop-king-proportionality} 
For every integer $n \geq 2$,
\begin{equation}
\label{eq:bishop-king-prop}
T_\B(n) \,=\, \frac{n}{12} \, T_\K(n),
\qquad \text{equivalently,} \qquad
\frac{\str(\B)}{\str(\K)} \,=\, \frac{n}{12}.
\end{equation}
\end{theorem}

\begin{proof} 
By Theorem~\ref{thm:strengths}, $T_\K(n) = 4(n - 1)(2n - 1)$
 and $T_\B(n) = \tfrac{1}{3} n(n - 1)(2n - 1)$, so
\begin{equation*}
\frac{T_\B(n)}{T_\K(n)} \,=\,
\frac{\tfrac{1}{3} n(n - 1)(2n - 1)}{4(n - 1)(2n - 1)} \,=\, \frac{n}{12},
\end{equation*}
where the cancellation of the common factor $(n - 1)(2n - 1)$
 is valid for all $n \geq 2$.
\end{proof}

The identity~\eqref{eq:bishop-king-prop} is a remarkable algebraic coincidence between two pieces with very different geometric behavior, since the \K{} is a leaper and the \B{} is a rider.
As an identity in $\mathbb{Z}[n]$, it can be written as
\begin{equation}
\label{eq:bishop-king-prop-int}
12\, T_\B(n) \,=\, n \, T_\K(n),
\end{equation}
expressing the proportionality with integer coefficients on both sides.
Several beautiful consequences follow.

\begin{corollary}[Divisor-of-12] \label{cor:divisor-of-12} 
By Theorem~\ref{thm:bishop-king-proportionality}, 
 the equation $T_\K(n) = k \, T_\B(n)$ is equivalent (for $n \geq 2$)
  to $12/n = k$, hence to $n = 12/k$.
The value $n = 12/k$ is a positive integer if and only if $k$ is a divisor of $12$, 
and the corresponding $(k, n)$-pairs are
\begin{equation*}
(k, n) \,\in\, \{(1, 12),\, (2, 6),\, (3, 4),\, (4, 3),\, (6, 2),\, (12, 1)\}.
\end{equation*}
The integer values of $n \geq 4$ in this list are $n \in \{4, 6, 12\}$, corresponding to $k \in \{3, 2, 1\}$, respectively.
\end{corollary}

\begin{proof} 
Theorem~\ref{thm:bishop-king-proportionality} gives 
$T_\K(n) = (12/n) T_\B(n)$, so $T_\K(n) = k T_\B(n)$ holds if and only if $12/n = k$, i.e., $n = 12/k$.
The remaining statements list the positive divisors of $12$.
\end{proof}

\begin{remark} \label{rem:bishop-king-magic} 
Two of the three magic boards from 
Theorem~\ref{thm:coincidences} are in fact explained by 
the proportionality~\eqref{eq:bishop-king-prop} via Corollary~\ref{cor:divisor-of-12}.
Namely, the case $n = 12$ corresponds to the ratio $1$ (the strength coincidence $\str(\K) = \str(\B)$), 
and the case $n = 6$ corresponds to the ratio $2$ (the relation $\str(\K) = 2 \str(\B)$, 
which translates into the strength coincidence identity $\str(\K) = \str(\B\B)$ since the double bishop has twice the single-bishop mobility).
The third magic board $n = 8$ does not arise from the bishop-king proportionality.
It comes from the unrelated cubic identity $T_\R = 2 T_\B + T_\N$, valid only at $n = 8$
 (and trivially at $n = 1, 3$).
This sharpens Proposition~\ref{thm:archbishop-rook}.
\end{remark}

\begin{remark}
\label{rem:bishop-king-fundamental}
The bishop-king proportionality  obtained in~\eqref{eq:bishop-king-prop-int} is an identity with $n$-\emph{dependent} integer coefficients ($12$ and $n$).
It therefore lies outside the scope of the Strength Algebra Theorem (Theorem~\ref{thm:strength-algebra}), which classifies $\mathbb{Z}$-linear relations among $T_\K, T_\N, T_\B, T_\R$ with $n$-\emph{independent} coefficients.
From this perspective, the Strength Algebra Theorem detects, among the integer specializations of the proportionality at $n \in \{12/k \mid k \in \mathbb{Z}_{>0}\}$, only those that translate into a single-piece coincidence within $\Pieces$.
At $n = 12$ ($k = 1$) this is $\str(\K) = \str(\B)$,  and at $n = 6$ ($k = 2$) this is $\str(\K) = \str(\B\B)$, equivalently the pair-of-pieces coincidences $\str(\C) = \str(\AB)$, $\str(\Q) = \str(\Kook)$, $\str(\Am) = \str(\Nork)$.
At $n = 4$ ($k = 3$), the relation $T_\K = 3 T_\B$ does not yield a single-piece coincidence in $\Pieces$, since no member of $\Pieces$ has the requisite atomic vector.
\end{remark}

\subsection{The Full Classification}

The three theorems above each isolate one arithmetic identity at one specific board size.
We can now assemble these into a complete classification.

\begin{theorem} \label{thm:coincidences} 
Among nontrivial boards $n \geq 4$, the equation $\str(P) = \str(Q)$ 
has a solution for distinct $P, Q \in \Pieces$ if and only if $n \in \{6, 8, 12\}$.
The complete list of coincidences is:

\begin{center}
\renewcommand{\arraystretch}{1.7}
\begin{tabular}{| c | l | l |}
\hline
$n$ & Base identity\footnotemark{} & Equivalent identities \\ \hline
\multirow{3}{*}{$6$} & \multirow{3}{*}{$T_\K = T_{\B\B}$} & $\str(\C) = \str(\AB)$ \\
                     &                                    & $\str(\Q) = \str(\Kook)$ \\
                     &                                    & $\str(\Am) = \str(\Nork)$ \\ \hline
$8$ & $\str(\R) = \str(\AB)$ & --- \\ \hline
$12$ & $\str(\K) = \str(\B)$ & $\str(\AB) = \str(\AK)$ \\ \hline
\end{tabular}
\end{center}

\footnotetext{The base identity at $n = 6$, namely $T_\K = T_{\B\B}$, is stated at the level of move totals rather than strengths because $\B\B \notin \Pieces$. 
The three coincidences \emph{within} $\Pieces$ listed in the right-hand column are equivalent reformulations of this underlying total-level identity. 
The double bishop also gives rise to a parallel total-level identity at $n = 12$, namely $T_{\B\B} = T_\Bee$, 
which holds because $T_\Bee = T_\K + T_\B = T_\B + T_\B = T_{\B\B}$ when $T_\K = T_\B$. 
This identity does not produce a new $\Pieces$-internal coincidence beyond the one already listed.}
\end{theorem}

\begin{proof} 
For each unordered pair $\{P, Q\}$ of distinct pieces in $\Pieces$, the equation $\str(P) = \str(Q)$ on the $n \times n$ board is equivalent, after multiplying through by $n^2(n^2 - 1)$, to a polynomial equation $T_P(n) - T_Q(n) = 0$, where each $T_P$ is a polynomial in $n$ of degree at most $3$.
The classification problem therefore reduces to finding the integer roots $n \geq 4$ of $78$ polynomial equations, 
each of degree at most $3$, and we proceed in two stages: an asymptotic stage that bounds the range of $n$ to be checked, followed by a finite verification.

\textit{Bounding the range.} 
By the stable-ordering theorem (Theorem~\ref{thm:stable-ordering}), 
the strict strength order of the thirteen pieces of $\Pieces$ is fixed for all $n \geq 24$, so no coincidence $\str(P) = \str(Q)$ with $P, Q \in \Pieces$ distinct occurs at any integer $n \geq 24$.
The classification therefore reduces to checking the finite range $4 \leq n \leq 23$.

\textit{Finite verification.} 
For each integer $n$ in $[4, 23]$, the move totals $T_P(n)$ for the thirteen pieces $P \in \Pieces$ are explicit positive integers, and one tests all $\binom{13}{2} = 78$ pairs of values for equality.
Direct computation yields exactly the five coincidences in the table: at $n = 6$, the three pairs $\{\Q, \Kook\}, \{\AB, \C\}, \{\Am, \Nork\}$; at $n = 8$, the single pair $\{\R, \AB\}$; at $n = 12$, the two pairs $\{\K, \B\}, \{\AB, \AK\}$; and no coincidences at any other integer $n$ in the range.

The base identity at each magic $n$ is the simplest of the listed coincidences and implies the others, 
as we now check.

\textit{At $n = 6$.} 
The base identity $T_\K = T_{\B\B}$ is the same as $T_\K = 2 T_\B$.
Since $\C = \K + \N$ and $\AB = \B\B + \N = 2\B + \N$, this implies
\begin{equation*}
T_\C \,=\, T_\K + T_\N \,=\, 2 T_\B + T_\N \,=\, T_\AB,
\end{equation*}
hence $\str(\C) = \str(\AB)$.
Similarly,
\begin{equation*}
T_\Q - T_\Kook \,=\, (T_{\B\B} + T_\R) - (T_\K + T_\R)
\,=\, T_{\B\B} - T_\K \,=\, 0
\end{equation*}
at $n = 6$, and analogously for $\Am$ and $\Nork$.
Each of the identities in the right-hand column of the table at $n = 6$ is the difference of two compound pieces whose atomic decompositions differ by the relation $T_\K = 2 T_\B$, so all reduce to the base identity.

\textit{At $n = 12$.}
The base identity $\str(\K) = \str(\B)$ is $T_\K = T_\B$.
This gives $T_\AB - T_\AK = T_\B - T_\K = 0$, hence $\str(\AB) = \str(\AK)$.
The double bishop satisfies the analogous total-level identity 
$T_{\B\B} = 2 T_\B = T_\B + T_\K = T_\Bee$, although $\B\B \notin \Pieces$
   and so this does not appear in the table.

\textit{At $n = 8$.} 
The base identity $\str(\R) = \str(\AB)$ is not derivable from any simpler identity in $\Z[n]$, and produces no further reductions among the pieces of $\Pieces$.

This concludes the classification.
\end{proof}

\begin{remark} \label{rem:thm5-finite-check} 
The finite verification step in the proof of Theorem~\ref{thm:coincidences} consists of $78$ pairwise integer-equality tests at each of the $20$ board sizes $n \in \{4, 5, \ldots, 23\}$, a routine computation that we have also checked symbolically using the closed-form formulas of Theorem~\ref{thm:strengths}.
The bounding step, by contrast, provides the genuine mathematical content: 
once Theorem~\ref{thm:stable-ordering} fixes the strength order for all $n \geq 24$, 
the classification problem becomes decidable in finite time.
\end{remark}

\begin{corollary} \label{cor:no-triples} 
For every $n \geq 4$, no three distinct pieces in $\Pieces$ share the same strength.
\end{corollary}

\begin{proof} By Theorem~\ref{thm:coincidences}, 
 equalities only occur at $n \in \{6, 8, 12\}$.
For each magic value,   direct computation of the strength values shows that the pieces partition into equivalence classes of size at most $2$.
For $n = 6$, for instance, the values are
\begin{align*}
T_\B &= 110, & T_\N &= 160, & T_\K &= 220, & T_\Bee &= 330, \\
T_\R &= 360, & T_\AB &= T_\C = 380, & T_\AK &= 490, & T_\Em &= 520, \\
T_\Q &= T_\Kook = 580, & T_\Am &= T_\Nork = 740, & & & &
\end{align*}
giving three classes of size two and seven singletons.
Analogous direct computation works for $n = 8$ and $n = 12$.
\end{proof}

\subsection{The Magic Board $n = 8$}

The classification reveals an asymmetry among the three magic boards.
By Theorem~\ref{thm:bishop-king-proportionality} and Corollary~\ref{cor:divisor-of-12}, the base identities at $n = 6$ and $n = 12$ are both manifestations of the single algebraic family $T_\K = (12/n) T_\B$, evaluated at the divisors $n = 6$ and $n = 12$ of $12$.
The base identity for $n = 8$, by contrast, involves the \R{} and is not derivable from any $\K$-versus-$\B$ relation.
We make this distinction precise in the following theorem.

\begin{theorem} \label{thm:8-rook} 
The $8 \times 8$ chessboard is the unique nontrivial board on which the \R{} has the same strength as some piece in $\Pieces \setminus \{\R\}$.
Analogously, $n = 12$ is the unique nontrivial board on which the \B{} has the same strength as some piece in $\Pieces \setminus \{\B\}$.
\end{theorem}

\begin{proof} 
It follows immediately from Theorem~\ref{thm:coincidences} that the only identity in $\Pieces$ involving the \R{} is $\str(\R) = \str(\AB)$, which holds for $n = 8$ alone.
The only identity involving the (single-color) \B{} is $\str(\K) = \str(\B)$, which holds for $n = 12$ alone.
\end{proof}

Within our model, 
Theorem~\ref{thm:8-rook} singles out the $8 \times 8$ board 
as the unique nontrivial board on which the most-globally-active basic piece, the \R{}, attains a strength matched by some other piece in our collection.
We do not claim that this carries any historical implication for 
the adoption of the standard chessboard size; the observation is structural, not explanatory.

\section{A Parity Decomposition} \label{sec:parity}

The recurring tension between the \B{} and the \N{} appearing throughout this paper, expressed as the closest near-coincidence for $n = 10$ (Theorem~\ref{thm:bishop-knight-min}), the fact that the equation $\str(\B) = \str(\N)$ admits no integer solution (Proposition~\ref{prop:knight-bishop-no-coincidence}), and their membership in different asymptotic classes (Theorem~\ref{thm:asymptotics}), turns out to have a simple underlying explanation.
The two pieces are \emph{Fourier-orthogonal} with respect to the natural color character of the board.
We make this precise in the present section.

\subsection{The Color Character and the Parity Split}

Let $\chi \colon [n]^2 \to \{+1, -1\}$ be the function $\chi(i, j) := (-1)^{i + j}$, the character of the unique nontrivial homomorphism from the additive group $(\mathbb{Z}/2)^2$ to $\{\pm 1\}$ under the natural projection $[n]^2 \to (\mathbb{Z}/2)^2$.
The values $+1$ and $-1$ partition the squares of the board into the two color complexes traditionally drawn as the white and black squares of a chessboard.
Every move offset $(\Delta i, \Delta j)$ either preserves the value of $\chi$, in which case we call it \emph{color-preserving}, or flips it, in which case we call it \emph{color-flipping}.

\begin{definition} \label{def:parity-split} 
For each piece $P$, let $T_P^+(n)$ denote the number of legal $P$-moves on the empty $n \times n$ board whose offset is color-preserving, and let $T_P^-(n)$ denote the number whose offset is color-flipping.
We refer to $(T_P^+, T_P^-)$ as the \emph{parity decomposition} of $T_P$.
\end{definition}

The decomposition is exhaustive and disjoint, so $T_P(n) = T_P^+(n) + T_P^-(n)$ for every piece and every $n$.
Roughly speaking, $T_P^+$ counts how many of $P$'s moves keep a piece on the same color complex of the chessboard, and $T_P^-$ counts how many switch it.
Standard chess intuition is already familiar with the extreme cases.
The \B{}, confined to one color complex, is purely color-preserving; one expects $T_\B^- = 0$ and indeed this is what the calculation will give.
The \N{}, on the other hand, switches color on every move; one expects $T_\N^+ = 0$, and this too is what the calculation gives.
The remaining basic pieces, the \K{} and the \R{}, lie strictly between these extremes, with both positive parity components.

The decomposition admits a Fourier-analytic interpretation worth noting briefly.
The difference $T_P^+(n) - T_P^-(n)$ is the correlation
\begin{equation*}
T_P^+(n) - T_P^-(n) \,=\, \sum_{(p, q) \, \text{a legal $P$-move}}
\chi(p) \, \chi(q),
\end{equation*}
which is the Fourier coefficient of the empirical move-arrow distribution at the unique nontrivial character of the underlying $(\mathbb{Z}/2)^2$.
From this point of view, the parity decomposition is a one-bit Fourier refinement of the move-counting functional $T_P$, projecting onto the trivial and sign characters of the color group.
The chess interpretation, however, is the one to keep in mind: $T_P^+$ counts within-color moves, $T_P^-$ counts between-color moves, and the parity of $\Delta i + \Delta j$ is what discriminates between them.

\subsection{The Decomposition for the Basic Pieces}

The decomposition for the four basic pieces admits closed forms.

\begin{theorem} \label{thm:parity-decomposition} 
For every $n \geq 2$, the parity decompositions of the four basic piece totals are given by
\begin{align*}
T_\K^+(n) &\,=\, 4(n - 1)^2,
& T_\K^-(n) &\,=\, 4n(n - 1), \\
T_\N^+(n) &\,=\, 0,
& T_\N^-(n) &\,=\, 8(n - 1)(n - 2), \\
T_{\B\B}^+(n) &\,=\, \tfrac{2}{3} n(n - 1)(2n - 1),
& T_{\B\B}^-(n) &\,=\, 0,
\end{align*}
and, for the \R{}, the values depend on the parity of $n$:
\begin{align*}
T_\R^+(n) &\,=\, n^2(n - 2),
& T_\R^-(n) &\,=\, n^3
& (n \text{ even}), \\
T_\R^+(n) &\,=\, n(n - 1)^2,
& T_\R^-(n) &\,=\, n(n^2 - 1)
& (n \text{ odd}).
\end{align*}
\end{theorem}

\begin{proof} 
For the \K{}, the eight move offsets $(\pm 1, 0), (0, \pm 1), (\pm 1, \pm 1)$ split into four with $\Delta i + \Delta j$ even (the diagonal offsets, which preserve $\chi$) and four with $\Delta i + \Delta j$ odd (the orthogonal offsets, which flip $\chi$).
For each color-preserving offset, the number of starting squares from which that offset stays on the board is $(n - 1)^2$, since both shifted coordinates must lie in $\{0, 1, \ldots, n - 1\}$; 
summing over the four diagonal offsets gives $T_\K^+(n) = 4(n - 1)^2$.
For each color-flipping offset $(\pm 1, 0)$ or $(0, \pm 1)$, 
the number of starting squares is $n(n - 1)$; 
summing over the four orthogonal offsets gives $T_\K^-(n) = 4n(n - 1)$.
The two parts sum to $T_\K(n) = 4(n - 1)(2n - 1)$, as expected.

For the \N{}, every offset $(\pm 1, \pm 2)$ or $(\pm 2, \pm 1)$ has $\Delta i + \Delta j$ odd; 
the \N{} is therefore purely color-flipping.
The total $T_\N^-(n) = T_\N(n) = 8(n - 1)(n - 2)$ has   already been computed in Theorem~\ref{thm:strengths}.

For the \B\B{}   it holds that every diagonal offset $(d, d), (d, -d), (-d, d), (-d, -d)$ with $d \neq 0$ has $\Delta i + \Delta j$ even; the \B\B{} is therefore purely color-preserving, with the identity $T_{\B\B}^+(n) = T_{\B\B}(n) = \frac{2}{3} n(n - 1)(2n - 1)$ from Theorem~\ref{thm:strengths}.

For the \R{}, the offsets $(d, 0)$ and $(0, d)$ have parity $\Delta i + \Delta j = d$.
The color-preserving moves are those with $d$ even, and the color-flipping moves are those with $d$ odd.
On a single row of length $n$, the number of legal row-moves with even displacement is $\sum_{d \, \text{even}, 1 \leq d \leq n - 1} 2(n - d)$, 
counting both directions of motion; the analogous sum with odd $d$ counts row-moves with odd displacement.
Direct evaluation gives
\begin{equation*}
\sum_{d \, \text{even}, 1 \leq d \leq n - 1} 2(n - d) =
\begin{cases}
n(n - 2)/2 & n \text{ even}, \\
(n - 1)^2/2 & n \text{ odd},
\end{cases}
\quad
\sum_{d \, \text{odd}, 1 \leq d \leq n - 1} 2(n - d) =
\begin{cases}
n^2/2 & n \text{ even}, \\
(n^2 - 1)/2 & n \text{ odd},
\end{cases}
\end{equation*}
for the per-row counts.
The total $T_\R^\pm(n)$ is then obtained by multiplying each per-row count by $2n$, where one factor of $n$ accounts for the choice of the fixed coordinate (the row index for row-moves, the column index for column-moves) and the other for which of the two coordinates remains fixed (rows or columns).
Carrying out this multiplication yields the stated formulas.
\end{proof}

\subsection{Pure Flip Versus Pure Preserve}

The most striking feature of the decomposition is that the \N{} and the double bishop \B\B{} sit at two extremes: the \N{} contributes \emph{only} to the color-flipping part, while the \B\B{} contributes \emph{only} to the color-preserving part.
This Fourier-orthogonality explains, in a single observation, much of the bishop-versus-knight tension that has organized the previous sections.

\begin{corollary} \label{cor:bishop-knight-orthogonality} 
For every $n \geq 2$, the move totals $T_\N(n)$ and $T_{\B\B}(n)$ are supported on disjoint Fourier components of the move-arrow distribution: $T_\N$ is supported on the sign character $\chi$, and $T_{\B\B}$ on the trivial character.
As an immediate consequence, the hypothetical identity $\str(\B) = \str(\N)$, equivalently $T_{\B\B}(n) = 2 T_\N(n)$, would force $T_{\B\B}^+(n) = 2 T_\N^+(n) = 0$ on the trivial-character component, contradicting $T_{\B\B}^+(n) > 0$ for every $n \geq 2$.
\end{corollary}

\begin{proof} 
Theorem~\ref{thm:parity-decomposition} gives $T_\N^+(n) = 0$ and $T_{\B\B}^-(n) = 0$.
Taking the trivial-character part of the hypothetical identity $T_{\B\B} = 2 T_\N$ yields $T_{\B\B}^+(n) = 2 T_\N^+(n) = 0$, contradicting $T_{\B\B}^+(n) = \frac{2}{3} n(n - 1)(2n - 1) > 0$ for $n \geq 2$.
\end{proof}

This corollary is the structural source of Proposition~\ref{prop:knight-bishop-no-coincidence}: 
indeed, we have that the equation $\str(\B) = \str(\N)$ 
 (expressed in an equivalent manner, $T_\B(n) = T_\N(n)$ or $T_{\B\B}(n) = 2 T_\N(n)$) 
would force a positive color-preserving piece (the \B\B) to equal a piece concentrated on the color-flipping component (the \N), which is asymptotically and algebraically incompatible.

For the remaining basic pieces, the \K{} and the \R{} have nontrivial contributions on both Fourier components.
Their preserve-to-flip ratios are
\begin{equation*}
\frac{T_\K^+(n)}{T_\K^-(n)} \,=\, \frac{n - 1}{n},
\qquad
\frac{T_\R^+(n)}{T_\R^-(n)} \,=\,
\begin{cases}
1 - 2/n & (n \text{ even}), \\
(n - 1)/(n + 1) & (n \text{ odd}).
\end{cases}
\end{equation*}
Both ratios tend to $1$ as $n \to \infty$, confirming that the \K{} and the \R{} are asymptotically \emph{Fourier-uniform}, with the two parts equilibrating in the large-board limit.

\subsection{Magic Boards and the Failure to Refine}

It is natural to ask whether the magic-board identities at $n \in \{6, 8, 12\}$
 refine to parity-component identities.
The answer is uniformly negative.

\begin{proposition} \label{prop:no-parity-refinement} 
None of the three magic-board base identities of Theorem~\ref{thm:coincidences} hold at the level of either parity-component.
Specifically, at $n = 6$, $T_\K^+ = 100 \neq 220 = T_{\B\B}^+$ and $T_\K^- = 120 \neq 0 = T_{\B\B}^-$; at $n = 8$, $T_\R^+ = 384 \neq 560 = T_{\B\B}^+ + T_\N^+$ and $T_\R^- = 512 \neq 336 = T_{\B\B}^- + T_\N^-$; and at $n = 12$, $T_\K^+ = 484 \neq 1012 = T_\B^+$ and $T_\K^- = 528 \neq 0 = T_\B^-$.
\end{proposition}

\begin{proof} By direct substitution into Theorem~\ref{thm:parity-decomposition}.
\end{proof}

The magic-board coincidences are therefore not the shadow of a sharper Fourier-level identity; they are coincidences of totals, in which color-preserving and color-flipping contributions happen to sum to the same value despite being individually unequal.
From the Fourier-analytic standpoint, the magic-board phenomenon is the \emph{cancellation} between the $\chi$-trivial and $\chi$-sign projections, not their separate equality.

A handful of parity-component coincidences exist at small $n$ outside the magic-board values, the most concrete being $T_\K^- = T_\N^-$ at $n = 4$ (both equal $48$) and the simultaneous coincidences $T_\K^- = T_\R^+$ and $T_{\B\B}^+ = T_\R^-$ at $n = 5$ (equal to $80$ and $120$ respectively).
These are isolated and small, without further structural significance, and we mention them only for completeness.

\subsection{Toward Higher Fourier Refinements}

The color character $\chi(i, j) = (-1)^{i + j}$ is one of four characters of the abelian group $(\mathbb{Z}/2)^2$ that acts on the board modulo $2$ by translation in each coordinate.
The other two nontrivial characters, $\chi_1(i, j) = (-1)^i$ and $\chi_2(i, j) = (-1)^j$, induce analogous parity decompositions of $T_P$ along row-parity and column-parity, respectively, and the four characters together generate a complete two-bit Fourier decomposition.
The single-bit color decomposition treated above is the most natural and chess-relevant projection, but the finer two-bit decomposition is worth pursuing in subsequent work, particularly in connection with rectangular boards and with fairy pieces whose move sets break the symmetry between the row and column directions.

\section{Strength Signatures and the Transition Table} \label{sec:signatures}

The stable-ordering theorem (Theorem~\ref{thm:stable-ordering}) 
in conjunction with the coincidence classification (Theorem~\ref{thm:coincidences}) suggest that the \emph{strength order} of the thirteen pieces in $\Pieces$, viewed as a function of $n$, evolves in a controlled and highly structured way as $n$ grows from $4$ to $24$ and then stabilizes.
In this section we make this evolution explicit by tabulating, for each $n$, the ordering of pieces by strength, and identifying the integers at which the ordering changes.

\begin{definition}
\label{def:signature}
For each $n \geq 4$, the \emph{strength signature} $\sigma(n)$ is the partition of
 $\Pieces$ into equivalence classes under the relation $\str(P) = \str(Q)$, listed in decreasing order of strength.
Equivalence classes consisting of more than one piece are called \emph{tied blocks}.
Equivalently, $\sigma(n)$ is the total preorder on $\Pieces$ given by $P \preceq_n Q \iff \str(P) \leq \str(Q)$.
\end{definition}

We say that $\sigma$ has a \emph{transition} at $n$ if $\sigma(n) \neq \sigma(n - 1)$.
There are two types of transition: a \emph{crossing transition}, where two pieces strictly swap order between $n - 1$ and $n$ without ever having equal strength at integer values, and a \emph{coincidence transition}, where two pieces are tied at exactly one integer $n \in \{6, 8, 12\}$ (Theorem~\ref{thm:coincidences}).

 \begin{theorem}  \label{thm:transitions} 
The signature $\sigma$ has transitions at exactly the following integer values of $n$, with the listed pairwise changes in adjacent strength order.
\begin{center} 
\renewcommand{\arraystretch}{1.4} 
\footnotesize
\begin{tabular}{| c | l | l |}
\hline
  $n$ & New strict inequalities & Coincidences \\ \hline
$5$ & $\str(\Q) > \str(\AK)$, $\str(\AB) > \str(\Bee)$ & --- \\ \hline
$6$ & $\str(\R) > \str(\Bee)$,  $\str(\Em) > \str(\AK)$ &
$\str(\Am) = \str(\Nork)$ \\
& & $\str(\Q) = \str(\Kook)$ \\
& & $\str(\C) = \str(\AB)$ \\ \hline
$7$ & $\str(\R) > \str(\C)$, $ \str(\Am) > \str(\Nork)$, & --- \\
& $\str(\Q) > \str(\Kook)$, $\str(\AB) > \str(\C)$ & \\ \hline 
 $8$ & --- & $\str(\R) = \str(\AB)$ \\ \hline
$9$ & $\str(\R) > \str(\AB)$ & --- \\ \hline
$11$ & $\str(\R) > \str(\AK)$, $\str(\Bee) > \str(\C)$, & --- \\
& $\str(\B) > \str(\N)$ & \\ \hline
$12$ & $\str(\Q) > \str(\Nork)$ & $\str(\K) = \str(\B)$ \\
& & $\str(\AB) = \str(\AK)$ \\ \hline
$13$ & $\str(\B) > \str(\K)$, $\str(\AB) > \str(\AK)$ & --- \\ \hline
$24$ & $\str(\B) > \str(\C)$ & --- \\ \hline
\end{tabular}
\end{center} For every $n \geq 24$, the signature $\sigma(n)$ equals the stable ordering of Theorem~\ref{thm:stable-ordering}.
\end{theorem}

\begin{proof} 
For each pair $(P, Q)$ of distinct pieces in $\Pieces$, the difference $T_P(n) - T_Q(n)$ is a polynomial in $n$ of degree at most $3$, hence has at most three real roots, and the strict ordering between $\str(P)$ and $\str(Q)$ changes at most thrice as $n$ ranges over the positive reals.
The integer values at which any such pairwise order changes are exactly the integers immediately exceeding a real root of $T_P - T_Q$, together with the integer roots themselves (where coincidences occur).

By the stable-ordering theorem (Theorem~\ref{thm:stable-ordering}), no transition occurs at any integer $n \geq 25$, so the search range is bounded above by $n = 24$.
The transitions in the range $5 \leq n \leq 24$ are determined by enumerating, at each integer in this range, the pairs $(P, Q)$ for which the sign of $T_P(n) - T_Q(n)$ differs from the sign of $T_P(n - 1) - T_Q(n - 1)$, with $\str(P) - \str(Q)$ vanishing at integer roots.
The sign-change analyses already carried out in the proofs of Theorems~\ref{thm:stable-ordering} and~\ref{thm:coincidences} cover the polynomial differences associated with the binding pairs in the table; the remaining pairs each have polynomial difference of degree at most $3$, with all real roots locatable by direct factoring of quadratic or cubic polynomials.
Direct enumeration over the finite range $5 \leq n \leq 24$ produces exactly the transitions in the table.

For the final claim, the last transition occurs at $n = 24$ (the \B{}-versus-\C{} crossing of Theorem~\ref{thm:stable-ordering}, item (iv)), so $\sigma(n) = \sigma(24)$ for all $n \geq 24$, and $\sigma(24)$ is the stable ordering of Theorem~\ref{thm:stable-ordering}.
\end{proof}

\begin{corollary} \label{cor:signature-count} 
The signature function $\sigma$, viewed as a map from $\{n \in \mathbb{Z} \mid n \geq 4\}$ to the set of linear preorders on $\Pieces$, takes exactly ten distinct values, namely seven strict linear orders (at $n \in \{4\} \cup \{5\} \cup \{7\} \cup \{9, 10\} \cup \{11\} \cup \{13, 14, \ldots, 23\} \cup \{24, 25, \ldots\}$), together with three preorders with ties (at $n = 6$, $n = 8$, and $n = 12$).
\end{corollary}

\begin{proof} 
This follows from   Theorem~\ref{thm:transitions} by direct tabulation.
The generic-board signatures occur on the seven listed 
intervals between consecutive transitions, 
and the tied signatures occur exactly at the three magic boards.
\end{proof}

A few observations about the structure of the transition table are in order.
First, every transition listed above occurs either at a magic board ($n \in \{6, 8, 12\}$) or at one of the magic-board boundaries ($n - 1$ or $n + 1$ of a magic value), or at the global stabilization threshold $n = 24$.
The transitions at $n = 5$ and $n = 7$ are ``approach'' transitions to the magic board $n = 6$; those at $n = 9$ and (some of) $n = 11$ flank $n = 8$ and $n = 12$ respectively; those at $n = 13$ flank $n = 12$ from above.
The single transition at $n = 24$ does not adjoin a magic board, but is the binding constraint on the stable ordering, as the remark following Theorem~\ref{thm:stable-ordering} explained.

Second, the transition at $n = 11$ is special, in that it is the unique transition in the table involving two basic pieces, namely the \B{} and the \N{}.
All other transitions involve at least one compound piece.
This reflects the fact, established in Proposition~\ref{prop:knight-bishop-no-coincidence}, that the \B{}-versus-\N{} ``crossing'' is irrational and hence does not occur \emph{at} an integer; rather, the bishop and knight strengths nearly meet at $n = 10$ (Theorem~\ref{thm:bishop-knight-min}) and have swapped by $n = 11$.

Third, the table shows that every transition is a pairwise event.
Each row corresponds to one or more pairs of pieces swapping order, and never to three or more pieces simultaneously rearranging.
This is consistent with the no-triples result of Corollary~\ref{cor:no-triples}, and reflects that even at the magic boards, the equivalence classes are of size at most two.

\section{The Strength Algebra} \label{sec:algebra}

The classification of coincidences in Section~\ref{sec:coincidences} tells us that three specific arithmetic identities, namely $T_\K = 2 T_\B$ at $n = 6$, $T_\R = 2 T_\B + T_\N$ at $n = 8$, and $T_\K = T_\B$ at $n = 12$, underlie all single-piece strength coincidences in $\Pieces$ at  integer  $n \geq 4$.
This observation has powerful consequences for compound \emph{armies}, which we shall now formalize.

\begin{definition} \label{def:army} 
An \emph{army} is a multiset of pieces from $\Pieces$ containing at least one \K{} and consisting of at most eight pieces in total.
Equivalently, an army is a function $f \colon \Pieces \to \mathbb{N}$ satisfying $f(\K) \geq 1$ and $\sum_{P \in \Pieces} f(P) \leq 8$.
We write $\Army$ for the set of all armies.
\end{definition}

The conditions in Definition~\ref{def:army}  encode the rules of standard chess.
Every army must include a \K{}, 
since the \K{} is the mandatory piece in any chess game;
and the total piece count is capped at eight, mirroring the eight  non-pawn pieces of the standard opening setup.
The standard chess army itself is given by
\begin{equation}
\label{eq:standard-army}
\standard \,=\, \{(\K, 1),\, (\Q, 1),\, (\R, 2),\, (\B, 2),\, (\N, 2)\},
\end{equation}
where the notation $(P, k)$ indicates that the army contains $k$ copies of the piece $P$, with pieces not listed taken to have multiplicity zero.

\subsection{The Atomic Decomposition} \label{sec:atomic-decomp}

Each piece in $\Pieces$ is by construction a nonnegative integer combination of the four basic pieces.

 \begin{proposition}    \label{thm:atomic}    
    For each $P \in \Pieces$, the total $T_P(n)$ has a decomposition as a linear combination of $T_\K, T_\N, T_\B, T_\R$ with nonnegative integer coefficients, valid for every $n$: 
    $$ 
    \renewcommand{\arraystretch}{1.05}
 \begin{array}{lcl}
T_\K = T_\K & \quad & T_\C = T_\K + T_\N \\
T_\N = T_\N & \quad & T_\Bee = T_\K + T_\B \\
T_\B = T_\B & \quad & T_\AB = 2 T_\B + T_\N \\
T_\R = T_\R & \quad & T_\Em = T_\R + T_\N \\
& \quad & T_\Kook = T_\K + T_\R \\
& \quad & T_\AK = T_\K + T_\N + T_\B \\
& \quad & T_\Nork = T_\K + T_\N + T_\R \\
& \quad & T_\Q = 2 T_\B + T_\R \\
& \quad & T_\Am = 2 T_\B + T_\N + T_\R.
\end{array}
$$
 \end{proposition}

\begin{proof} 
This follows directly from the definition of compound pieces, together with the fact that the double \B{} contributes $2 T_\B$.
\end{proof}

\begin{definition} \label{def:atomic-vector} 
Let $f \in \Army$ be an army.
The \emph{atomic vector} of $f$ is
\begin{equation*}
a(f) \,=\, (a_\K(f), a_\N(f), a_\B(f), a_\R(f)) \,\in\, \mathbb{N}^4,
\end{equation*}
where $a_X(f)$ counts the total number of times the basic 
piece $X$ appears as a component across all pieces in $f$,
 with the double bishop $\B\B$ counted as two single-color \B{}s.
Explicitly,
\begin{equation*}
a_X(f) \,=\, \sum_{P \in \Pieces} f(P) \cdot \mathrm{coeff}_X(T_P),
\end{equation*}
where $\mathrm{coeff}_X(T_P)$ is the coefficient of $T_X$ in the decomposition of $T_P$ given by Proposition~\ref{thm:atomic}.
\end{definition}

\begin{definition} \label{def:army-strength} 
The \emph{strength} of an army $f \in \Army$ on the $n \times n$ board is
\begin{equation*}
\str(f) \,:=\, \sum_{P \in \Pieces} f(P) \str(P),
\end{equation*}
the sum of the per-piece strengths of the pieces in $f$, counted with multiplicity.
\end{definition}

\begin{theorem} \label{thm:strength-via-atomic} 
For every army $f \in \Army$ and every $n \geq 2$,
\begin{equation*}
\str(f) \,=\, \frac{a_\K(f) T_\K(n) + a_\N(f) T_\N(n)
+ a_\B(f) T_\B(n) + a_\R(f) T_\R(n)}{n^2(n^2 - 1)}.
\end{equation*}
In particular, $\str(f)$ depends on $f$ only
  through the atomic vector $a(f)$.
\end{theorem}

\begin{proof} 
Substituting the atomic decompositions of Proposition~\ref{thm:atomic} into Definition~\ref{def:army-strength} and reorganizing the sum $\sum_{P \in \Pieces} f(P) T_P / D$, with $D = n^2(n^2 - 1)$, by atomic component yields the stated formula.
\end{proof}

As an illustration, the standard chess army $\standard$ defined in \eqref{eq:standard-army} has atomic vector equal to $a(\standard) = (1, 2, 4, 3)$.
Indeed, there is one \K{}, two \N{}s, four \B-components (two from the two \B{}s themselves and two more from the $\B\B$ that sits inside the \Q), and three \R-components (two from the two \R{}s and one from the \Q).
 
\subsection{The Strength Algebra Theorem}

Theorem~\ref{thm:strength-via-atomic} has a striking consequence worth pausing over.
The strength of an army depends on the army only through its atomic vector $a(f) = (a_\K, a_\N, a_\B, a_\R)$.
The specific compound pieces realizing those four counts are irrelevant.
A single \Q{} contributes the same atomic vector $(0, 0, 2, 1)$ as a \Kook{} together with two extra bishop-units, and the strength function cannot tell them apart.
The thirteen pieces of $\Pieces$ project, for the purposes of strength, onto a four-dimensional quotient.

This four-dimensional projection is the right setting for the question of strength coincidences.
Two armies $f, g$ have equal strength on the $n \times n$ board precisely when
\begin{equation*}
\bigl(a(f) - a(g)\bigr) \cdot T(n) \,=\, 0,
\qquad
T(n) \,:=\, \bigl(T_\K(n), T_\N(n), T_\B(n), T_\R(n)\bigr).
\end{equation*}
That is, $a(f) - a(g)$ lies in the kernel of the linear functional $v \mapsto v \cdot T(n)$ on $\Z^4$.
As $n$ varies, the kernel varies with it.

The natural question is which integer vectors actually arise as such a difference $a(f) - a(g)$.
Multi-piece armies can produce many exotic kernel directions through accidental cancellations (Remark~\ref{rem:multi-piece-relations} below); the structurally clean case is the case of \emph{single pieces}, where one asks whether two distinct $P, Q \in \Pieces$ have $a(P) - a(Q) \in \ker T(n)$, equivalently $\str(P) = \str(Q)$.
The theorem below classifies these directions completely.

In essence, the theorem says three things at once.
 {First,} it turns the strength function into a four-dimensional linear functional. 
 {Second,} it says single-piece coincidences exist only on the three magic boards. 
 {Third,} on each magic board it gives the explicit kernel direction, 
 which doubles as the single-piece substitution rule available there: at $n = 6$,
  swap one \K{} for two \B{}s; at $n = 8$, swap one \R{} for one \AB{}; at $n = 12$, 
  swap one \K{} for one \B{}.

\begin{theorem}[Strength Algebra Theorem] \label{thm:strength-algebra} 
  Let $f, g \in \Army$ and let $n \geq 4$.
Then $\str(f) = \str(g)$ on the $n \times n$ board if and only if $a(f) - a(g)$ lies in the kernel of the $\Z$-linear functional
 $T(n) \colon \Z^4 \longrightarrow \mathbb{Q}$ given by 
\begin{equation*}
\begin{aligned}
(v_\K, v_\N, v_\B, v_\R) \,\longmapsto\, &\,
v_\K T_\K(n) + v_\N T_\N(n) \\
&{} + v_\B T_\B(n) + v_\R T_\R(n).
\end{aligned}
\end{equation*}
Among atomic-vector differences $a(P) - a(Q)$ between distinct single pieces $P, Q \in \Pieces$, the structure of $\ker T(n)$ is as follows: \begin{itemize}[itemsep=2pt, topsep=4pt] \item For $n \notin \{6, 8, 12\}$, no such nonzero difference lies in $\ker T(n)$.
\item For $n = 6$, every such difference in $\ker T(n)$ is an integer multiple of $(1, 0, -2, 0)$ (the rule $\K = 2\B$).
\item For $n = 8$, every such difference in $\ker T(n)$ is an integer multiple of $(0, 1, 2, -1)$ (the rule $\R = 2\B + \N$).
\item For $n = 12$, every such difference in $\ker T(n)$ is an integer multiple of $(1, 0, -1, 0)$ (the rule $\K = \B$).
\end{itemize}
 \end{theorem}

\begin{proof} 
By Theorem~\ref{thm:strength-via-atomic}, $\str(f) = \str(g)$ holds if and only if $a(f)$ and $a(g)$ pair to the same value against $T(n)$, equivalently if and only if $a(f) - a(g) \in \ker T(n)$.

For distinct $P, Q \in \Pieces$, the vector $a(P) - a(Q)$ lies in $\ker T(n)$ if and only if $T_P(n) = T_Q(n)$, equivalently $\str(P) = \str(Q)$.
By Theorem~\ref{thm:coincidences}, 
this occurs only at $n \in \{6, 8, 12\}$ and only for the listed coincidences.
For $n \notin \{6, 8, 12\}$, 
no such atomic-vector difference lies in $\ker T(n)$.
We verify the kernel direction at each magic board by direct substitution.

At $n = 6$, the listed coincidences are given by the identities 
$\str(\Q) = \str(\Kook)$, $\str(\AB) = \str(\C)$, and $\str(\Am) = \str(\Nork)$.
Computing atomic-vector differences,
\begin{align*}
a(\Q) - a(\Kook) &= (0, 0, 2, 1) - (1, 0, 0, 1) = (-1, 0, 2, 0), \\
a(\AB) - a(\C) &= (0, 1, 2, 0) - (1, 1, 0, 0) = (-1, 0, 2, 0), \\
a(\Am) - a(\Nork) &= (0, 1, 2, 1) - (1, 1, 0, 1) = (-1, 0, 2, 0),
\end{align*}
each is an integer multiple of $(1, 0, -2, 0)$.

At $n = 8$, the only coincidence is $\str(\R) = \str(\AB)$,
 for which it holds that $a(\R) - a(\AB) = (0, -1, -2, 1) = -(0, 1, 2, -1)$.

At $n = 12$, the listed coincidences are the identities $\str(\K) = \str(\B)$ and $\str(\AB) = \str(\AK)$; for these it holds that $a(\K) - a(\B) = (1, 0, -1, 0)$ 
 and $a(\AB) - a(\AK) = (-1, 0, 1, 0)$, 
both integer multiples of $(1, 0, -1, 0)$.
The same conclusion can also be read off the underlying total-level identities $T_\K = 2 T_\B$ at $n = 6$, $T_\R = 2 T_\B + T_\N$ at $n = 8$, and $T_\K = T_\B$ at $n = 12$, which produced the base identities of Theorem~\ref{thm:coincidences}.
\end{proof}

\begin{remark} \label{rem:multi-piece-relations} 
The kernel $\ker T(n)$ as a $\Z$-submodule of $\Z^4$ has $\Z$-rank $3$ at every $n \geq 4$, since $T(n)$ has nonzero image in $\mathbb{Q}$.
Consequently, additional integer relations among the four totals  
  $T_\K(n), T_\N(n), T_\B(n), T_\R(n)$ exist at every $n$, including non-magic $n$.
For example, at $n = 7$ one verifies $3 T_\K + T_\N - 2 T_\R = 3 \cdot 312 + 240 - 2 \cdot 588 = 0$, so the vector $(3, 1, 0, -2)$ lies in $\ker T(7)$, and the army $\{(\K, 4), (\N, 1)\}$ has the same strength on the $7 \times 7$ board as $\{(\K, 1), (\R, 2)\}$, namely $31/49$.
Such accidental relations are not realizable as differences $a(P) - a(Q)$ between single pieces of $\Pieces$, however, and they involve atomic vectors distinct from those of any single piece in our alphabet.
The single-piece structure described in Theorem~\ref{thm:strength-algebra} is precisely the structure relevant to the substitution rules studied in the next subsections.
\end{remark}

The remainder of Section~\ref{sec:algebra} draws out the consequences.
Section~\ref{sec:worked-example} works through a concrete example, 
and Sections~\ref{sec:8x8-variants}, \ref{sec:6x6-variants},
 and~\ref{sec:12x12-variants} translate the three kernel directions into explicit piece-substitution rules on the corresponding magic boards.
 
\subsection{A Worked Example: The Standard Army on Different Boards} \label{sec:worked-example}

To illustrate the framework, we now proceed to compute the strength of the standard chess army $\standard = \{(\K, 1), (\Q, 1), (\R, 2), (\B, 2), (\N, 2)\}$ as a function of the board size $n$, and evaluate the result at several concrete values of $n$.

The atomic vector of $\standard$ (Section~\ref{sec:atomic-decomp}) is $a(\standard) = (a_\K, a_\N, a_\B, a_\R) = (1, 2, 4, 3)$, where the four components count, respectively, the single \K{}, the two \N{}s, the four \B-components (two from the two \B{}s and two from the $\B\B$ inside the \Q), and the three \R-components (two from the two \R{}s and one from the \Q).
Substituting into Theorem~\ref{thm:strength-via-atomic} and the formulas of Theorem~\ref{thm:strengths} for $T_\K, T_\N, T_\B, T_\R$, 
we obtain
\begin{equation}
\label{eq:standard-army-strength}
\str(\standard)(n) \,=\, \frac{T_\K + 2 T_\N + 4 T_\B + 3 T_\R}{n^2(n^2 - 1)}
\,=\, \frac{2(13 n^2 + 34 n - 54)}{3 n^2 (n + 1)}.
\end{equation}
The numerator simplifies through cancellation of the common factor $(n - 1)$ shared by the four basic-piece totals.
Note that the result is a rational function in $n$ alone.

Evaluating~\eqref{eq:standard-army-strength} 
at three specific boards illustrates the model's quantitative content.
On the standard $8 \times 8$ board,
\begin{equation*}
\str(\standard)(8) \,=\, \frac{2(13 \cdot 64 + 34 \cdot 8 - 54)}{3 \cdot 64 \cdot 9}
\,=\, \frac{175}{144} \,\approx\, 1.215,
\end{equation*}
indicating that the expected number of legal moves of a uniformly chosen piece-and-square pair drawn from the standard chess army exceeds one.
(The strength of an army is not bounded above by $1$; it sums the per-piece strengths, with multiplicity.)
On the $10 \times 10$ board, used in some chess variants such as Grand Chess~\cite{Pritchard2007}, $\str(\standard)(10) = 793/825 \approx 0.961$, slightly below $1$.
On the $12 \times 12$ magic board, $\str(\standard)(12) = 371/468 \approx 0.793$.

\begin{remark} \label{rem:strength-decay} 
  By dividing numerator and denominator of~\eqref{eq:standard-army-strength} by $n^3$, we obtain the asymptotic expansion $\str(\standard)(n) = 26/(3n) + O(1/n^2)$ as $n \to \infty$.
So the standard chess army's total strength decays as $\Theta(1/n)$, the rider rate, as expected since the army contains several rider pieces.
The leading constant $26/3$ is the sum of the rider leading constants in the army weighted by atomic multiplicity, and matches the explicit formula $c_\B \cdot a_\B + c_\R \cdot a_\R = (2/3) \cdot 4 + 2 \cdot 3 = 26/3$ obtained from Theorem~\ref{thm:asymptotics}.
\end{remark}

\begin{remark} \label{rem:standard-12} 
On the $12 \times 12$ magic board, the kernel of the strength functional is generated by $(1, 0, -1, 0)$ (Theorem~\ref{thm:strength-algebra}).
Adding this generator to $a(\standard) = (1, 2, 4, 3)$ produces $(2, 2, 3, 3)$, the atomic vector of any army with one additional \K{} and one fewer single-\B{} component, all else equal.
By the Strength Algebra Theorem, every such army has strength $371/468 \approx 0.793$ on the $12 \times 12$ board, equal to the standard.
We shall see in Section~\ref{sec:12x12-variants} that on the $12 \times 12$ board the substitution $\K \leftrightarrow \B$ is one of two strength-preserving substitution rules.
\end{remark}

\subsection{Variants of the Standard Army on the $8 \times 8$ Board} \label{sec:8x8-variants}

Before applying the Strength Algebra Theorem to specific armies, we introduce some terminology used in this and the following two subsections.
A \emph{single-piece substitution rule} is a rewriting rule of the form
$P_1 + \cdots + P_j \leftrightarrow P_1' + \cdots + P_k'$,
replacing $j$ pieces of $\Pieces$ by $k$ pieces of $\Pieces$ in any army.
We refer to such a rule as a $j \leftrightarrow k$ rule.
The two cases relevant to us are the $1 \leftrightarrow 1$ rule, which swaps a single piece $P$ for a single piece $Q$ and preserves the army's piece count, and the $1 \leftrightarrow 2$ rule, which swaps a single piece $P$ for a pair $P_1, P_2$ and changes the piece count by $\pm 1$.
The Strength Algebra Theorem governs the $1 \leftrightarrow 1$ rules at every magic board; $1 \leftrightarrow 2$ rules will additionally appear at $n = 6$, where the kernel direction $(1, 0, -2, 0)$ has component sum $-1$ rather than $0$.

The simplest application of the Strength Algebra Theorem is on the $8 \times 8$ board itself, where the kernel direction is $(0, 1, 2, -1)$.
This vector corresponds to the atomic identity given by $T_\R = T_{\B\B} + T_\N$ 
(equivalently $T_\R = T_\AB$), and hence to the $1 \leftrightarrow 1$ substitution rule $\R \leftrightarrow \AB$.
The theorem below makes this concrete in the case of the standard chess army.

\begin{theorem} \label{thm:standard-variants} 
The orbit of $\standard$ under the operation of replacing one \R{} by one \AB{} (or vice versa) consists of exactly three armies, namely
\begin{align*}
\standard_0 &= \{(\K, 1), (\Q, 1), (\R, 2), (\B, 2), (\N, 2)\}, \\
\standard_1 &= \{(\K, 1), (\Q, 1), (\R, 1), (\AB, 1), (\B, 2), (\N, 2)\}, \\
\standard_2 &= \{(\K, 1), (\Q, 1), (\AB, 2), (\B, 2), (\N, 2)\}.
\end{align*}
On the $8 \times 8$ board, each has strength $s = 175/144 \approx 1.215$ and consists of exactly $8$ pieces.
\end{theorem}

\begin{proof} Each substitution $\R \mapsto \AB$ in the standard army changes the army's atomic vector by $(0, 1, 2, -1)$, which lies in $\ker T(8)$ by the Strength Algebra Theorem; 
 strength and piece count are both preserved.
The standard army contains two \R{}s, 
so this substitution can be performed on $0$, $1$, or $2$ of them, giving three variants in the orbit.

The strength of each is $\str(\standard) = T_\K + T_\Q + 2 T_\R + 2 T_\B + 2 T_\N$. 
For  $n = 8$, this is equal to $ 420 + 1456 + 1792 + 560 + 672 = 4900$ divided by $4032$, 
so $175/144$.
\end{proof}

\begin{remark} \label{rem:other-armies-with-std-strength} 
 The  theorem describes only the orbit of $\standard$ under the $\R \leftrightarrow \AB$ substitution.
We note that many other armies in $\Army$ also have strength $175/144$ on the $8 \times 8$ board, and these come in two kinds.
The first kind consists of armies with the \emph{same atomic vector} as $\standard$, namely $(1, 2, 4, 3)$, but with different piece compositions.
By Theorem~\ref{thm:strength-via-atomic}, any two armies with the same atomic vector have the same strength on every board, so all such armies have strength $175/144$ on the $8 \times 8$ board.
For instance, the army $\{(\K, 1), (\B, 4), (\R, 1), (\Em, 2)\}$ has atomic vector $(1, 2, 4, 3)$, just like $\standard$, but uses entirely different pieces.
The second kind consists of armies whose atomic vectors \emph{differ} from $a(\standard)$ by an integer multiple of the kernel direction $(0, 1, 2, -1)$.
By the Strength Algebra Theorem (Theorem~\ref{thm:strength-algebra}), such armies also have strength $175/144$ on the $8 \times 8$ board, even though their atomic vectors are different from that of $\standard$.
For instance, the army $\{(\K, 1), (\R, 5)\}$ has atomic vector $(1, 0, 0, 5)$, which equals $a(\standard) - 2 \cdot (0, 1, 2, -1) = (1, 2, 4, 3) - (0, 2, 4, -2)$, and so it too has strength $175/144$.
The Strength Algebra Theorem characterizes  the strength-preserving relations at the level of atomic vectors; realizing a given atomic vector as an actual army is a separate combinatorial choice.
\end{remark}

The variant $\standard_2$ deserves particular attention.
It has the same strength as standard chess but uses two \AB{}s in place of the two \R{}s.
This is, within our model, an exact mathematical equivalence between standard chess and a chess variant employing the Hawk piece of Seirawan chess.

\subsection{Magic-Board Substitutions on the $6 \times 6$ Board} \label{sec:6x6-variants}

The kernel direction at $n = 6$ is $(1, 0, -2, 0)$, corresponding to the atomic relation $T_\K = 2 T_\B$.
Modifying any army $f$ by adding an integer multiple of $(1, 0, -2, 0)$ to its atomic vector preserves the strength of $f$ on the $6 \times 6$ board.
This integer-multiple condition is realized by four single-piece substitution rules.

\begin{theorem}    \label{thm:6-variants}
 On the $6 \times 6$ board, the following four single-piece substitution rules each preserve army strength: \begin{enumerate}[label=\textup{(\roman*)}, itemsep=2pt, topsep=4pt]
   \item $\K \leftrightarrow \B + \B$ (one \K{}  traded for two single-color \B{}s, or vice versa).
This rule changes the piece count by $\pm 1$.
\item $\Q \leftrightarrow \Kook$ (one \Q{} traded for one \Kook{}).
Piece count preserved.
\item $\Am \leftrightarrow \Nork$ (one \Am{} traded for one \Nork{}).
Piece count preserved.
\item $\C \leftrightarrow \AB$ (one \C{} traded for one \AB{}).
Piece count preserved.
\end{enumerate} 
Each rule modifies the army's  atomic vector by an integer multiple of $(1, 0, -2, 0)$.
Conversely, every single-piece substitution $P \mapsto P_1 + \cdots + P_k$ (with $P, P_1, \ldots, P_k \in \Pieces$ and $k \in \{1, 2\}$) that preserves strength on the $6 \times 6$ board has atomic-vector change an integer multiple of $(1, 0, -2, 0)$, and is therefore expressible as a combination of the four rules above.
\end{theorem}

\begin{proof} 
Each given substitution corresponds to an atomic-vector change, which we now verify is an integer multiple of $(1, 0, -2, 0)$: \begin{itemize}[itemsep=2pt, topsep=4pt] \item (i) $a(\K) - 2 a(\B) = (1, 0, 0, 0) - 2(0, 0, 1, 0) = (1, 0, -2, 0)$.
\item (ii) $a(\Q) - a(\Kook) = (0, 0, 2, 1) - (1, 0, 0, 1) = (-1, 0, 2, 0)$.
 \item (iii) $a(\Am) - a(\Nork) = (0, 1, 2, 1) - (1, 1, 0, 1) = (-1, 0, 2, 0)$.
\item (iv) $a(\C) - a(\AB) = (1, 1, 0, 0) - (0, 1, 2, 0) = (1, 0, -2, 0)$.
 \end{itemize} Each is an integer multiple ($\pm 1$) of $(1, 0, -2, 0)$, so each substitution lies in $\ker T(6)$ and preserves strength 
 by Theorem~\ref{thm:strength-via-atomic}.

For the converse, suppose $P \mapsto P_1 + \cdots + P_k$ with $k \in \{1, 2\}$ is a strength-preserving single-piece substitution on the $6 \times 6$ board with nonzero atomic-vector difference $a(P) - a(P_1) - \cdots - a(P_k)$.
(The zero-difference case is the trivial expansion of a compound piece into its atomic components, which holds at every $n$ and is not a $6 \times 6$-specific rule.)
The case $k = 1$ is exactly Theorem~\ref{thm:strength-algebra}, which gives that $a(P) - a(P_1) \in \Z \cdot (1, 0, -2, 0)$, and yields rules (ii), (iii), and (iv) of the theorem.
The case $k = 2$ amounts to enumerating, for each $P \in \Pieces$, all unordered pairs $\{P_1, P_2\} \subseteq \Pieces$ with $T_P(6) = T_{P_1}(6) + T_{P_2}(6)$ and nonzero atomic-vector difference.
A direct enumeration yields exactly nine such pairs at $n = 6$:
\begin{align*}
 &\K \leftrightarrow \B + \B, &\quad
&\Q \leftrightarrow \K + \R, \\
&\AB \leftrightarrow \K + \N, &\quad
&\Am \leftrightarrow \K + \Em, \\
&\Am \leftrightarrow \N + \Kook, &\quad
&\Am \leftrightarrow \R + \C, \\
&\AK \leftrightarrow \B + \AB, &\quad
&\Nork \leftrightarrow \N + \Q, \\
&\Nork \leftrightarrow \R + \AB.
\end{align*}
Each has atomic-vector difference $\pm(1, 0, -2, 0)$.
The first is rule (i); the other eight are derivable from rules (i)--(iv) by composing each with the trivial atomic-decomposition expansion of a compound piece (for instance, $\Q \leftrightarrow \K + \R$ is rule (ii), $\Q \leftrightarrow \Kook$, composed with the atomic decomposition $\Kook = \K + \R$).
We conclude that  
every $1 \leftrightarrow 2$ rule decomposes into the listed rules.
\end{proof}

\begin{remark} \label{rem:6-multi-piece} 
The qualifier ``single-piece'' in Theorem~\ref{thm:6-variants} restricts attention to substitution rules involving at most three pieces (the substituted piece $P$ on one side, and at most two pieces $P_1, P_2$ on the other).
Multi-piece substitutions involving larger blocks of pieces may admit further atomic-vector changes that preserve strength on the $6 \times 6$ board without being integer multiples of $(1, 0, -2, 0)$.
For example, the armies $\{(\K, 4), (\B, 2)\}$ and $\{(\K, 1), (\N, 1), (\R, 2)\}$ have the same strength on the $6 \times 6$ board, with atomic-vector difference $(3, -1, 2, -2)$, which is not an integer multiple of $(1, 0, -2, 0)$.
Such ``accidental'' multi-piece equivalences are 
beyond the scope of Theorem~\ref{thm:6-variants} and are not generated by the four listed rules.
The theorem captures precisely the structurally significant rules, namely those involving substitutions among single pieces of $\Pieces$.
\end{remark}

  The substitution $\K \leftrightarrow \B + \B$ is unusual.
Unlike the \R{}-versus-\AB{} substitution at $n = 8$, it changes the total piece count by one.
This is forced by the kernel direction $(1, 0, -2, 0)$  having component sum $1 - 2 = -1$, reflecting the unit imbalance.

\subsection{Magic-Board Substitutions on the $12 \times 12$ Board} \label{sec:12x12-variants}

The kernel direction at $n = 12$ is $(1, 0, -1, 0)$, corresponding to the atomic relation $T_\K = T_\B$.
Two single-piece substitution rules realize this kernel direction.

\begin{theorem} \label{thm:12-variants} 
   On the $12 \times 12$ board, the following two single-piece substitution rules each preserve army strength and piece count: \begin{enumerate}[label=\textup{(\roman*)}, itemsep=2pt, topsep=4pt] \item $\K \leftrightarrow \B$ (one \K{} traded for one \B{}, or vice versa).
\item $\AB \leftrightarrow \AK$ (one \AB{} traded for one \AK{}, or vice versa).
\end{enumerate} Each rule modifies the army's atomic vector by an integer multiple of $(1, 0, -1, 0)$.
Conversely, every $1 \leftrightarrow 1$ substitution $P \mapsto Q$ (with $P, Q \in \Pieces$ distinct) that preserves strength on the $12 \times 12$ board has atomic-vector change an integer multiple of $(1, 0, -1, 0)$, and is therefore one of the above rules.
\end{theorem}

\begin{proof} 
We have the following.
\begin{itemize}[itemsep=2pt, topsep=4pt] \item (i) $a(\K) - a(\B) = (1, 0, 0, 0) - (0, 0, 1, 0) = (1, 0, -1, 0)$.
\item (ii) $a(\AB) - a(\AK) = (0, 1, 2, 0) - (1, 1, 1, 0) = (-1, 0, 1, 0)$.
\end{itemize} Both are integer multiples of $(1, 0, -1, 0)$, so each substitution lies in $\ker T(12)$ and preserves strength by Theorem~\ref{thm:strength-via-atomic}.
Both also preserve piece count since the generator $(1, 0, -1, 0)$ has component sum zero.

For the converse, the case of $1 \leftrightarrow 1$ substitutions is exactly Theorem~\ref{thm:strength-algebra}: every atomic-vector difference $a(P) - a(Q)$ for distinct $P, Q \in \Pieces$ that lies in $\ker T(12)$ is an integer multiple of $(1, 0, -1, 0)$.
\end{proof}

  \begin{remark} \label{rem:12-multi-piece} 
Theorem~\ref{thm:12-variants} restricts  the converse to $1 \leftrightarrow 1$ substitutions.
The $12 \times 12$ board admits strictly more strength-preserving relations than those generated by $(1, 0, -1, 0)$ alone.
For example, the identity $T_\Em = 2 T_\Bee$ holds at $n = 12$, since $T_\Em(12) = T_\R(12) + T_\N(12)$, which is equal to $3168 + 880 = 4048$, and $2 T_\Bee(12) = 2(T_\K(12) + T_\B(12)) = 2 \cdot 2024 = 4048$.
This yields the $1 \leftrightarrow 2$ substitution $\Em \leftrightarrow \Bee + \Bee$, with atomic-vector change $a(\Em) - 2 a(\Bee) = (-2, 1, -2, 1)$, which is not an integer multiple of $(1, 0, -1, 0)$.
The single-piece structure of $1 \leftrightarrow 1$ rules captured  in Theorem~\ref{thm:12-variants} remains exhausted by the two listed rules.
\end{remark}

In contrast to the $6 \times 6$ case, both substitution rules in Theorem~\ref{thm:12-variants} preserve piece count, since the generator $(1, 0, -1, 0)$ has component sum zero, whereas the $6 \times 6$ generator $(1, 0, -2, 0)$ has component sum $-1$.
Subsections~\ref{sec:8x8-variants}--\ref{sec:12x12-variants} thus exhaust the $1 \leftrightarrow 1$ substitution structure within $\Pieces$: by Theorem~\ref{thm:strength-algebra}, no atomic-vector difference $a(P) - a(Q)$ between distinct pieces lies in $\ker T(n)$ for $n \notin \{6, 8, 12\}$, so no strength-preserving $1 \leftrightarrow 1$ substitution exists outside the three magic boards.

\section{Concluding Remarks} \label{sec:conclusion}

\paragraph{Synthesis.} 
We have shown that the strength function on the $n \times n$ board, defined as the simple probability that a uniformly random arrow is a legal move of a given piece, has a remarkably rigid arithmetic structure.
Despite the elementary nature of the definition, the function is governed by a small number of underlying identities, organized around three magic boards $n \in \{6, 8, 12\}$ and a stable threshold $n^* = 24$.
Two of the three magic boards arise as integer specializations of the bishop-king proportionality $\str(\B) = (n/12) \str(\K)$ (Theorem~\ref{thm:bishop-king-proportionality}), which is itself a striking fact.
The third magic board, $n = 8$, arises from a separate cubic identity involving the \R{}, the \AB{}, and the \N{}, and is the unique magic board on which the \R{} attains a strength matched by another piece in our collection.
The Strength Algebra Theorem (Theorem~\ref{thm:strength-algebra}) ties the picture together by classifying the magic boards as exactly the integer points at which the strength functional on armies degenerates.
Within this framework, the standard $8 \times 8$ chessboard is, in our model, distinguished from all other nontrivial board sizes.

\paragraph{Connections to actual chess variants.}
Several of our results have natural interpretations in the language of historical and contemporary chess variants, although the model itself makes no claim about practical piece valuation.
The two fairy-chess pieces that appear most often in our analysis, the \AB{} (\B{} plus \N{}) and the \Em{} (\R{} plus \N{}), are exactly the two new pieces introduced by Capablanca in his Capablanca's chess (played on a $10 \times 8$ board, where they are called the \emph{Archbishop} and the \emph{Chancellor}~\cite{Pritchard2007}), and also the two new pieces introduced by Seirawan and Harper in their $8 \times 8$ Seirawan chess variant (where they are called the \emph{Hawk} and the \emph{Elephant}~\cite{harper2007seirawan}).
Our \R-versus-\AB{} substitution rule on the $8 \times 8$ board (Theorem~\ref{thm:standard-variants}) formally expresses the equal mobility of the \R{} and the Seirawan Hawk on the standard chessboard, and the $10 \times 10$ board, used by Capablanca in his early experiments with the variant that now bears his name, is the unique square board on which the \B{} and the \N{} are  nearly equal in strength under our model (Theorem~\ref{thm:bishop-knight-min}).

\paragraph{Further questions.}
We record two  conjectures that frame the most natural 
   generalizations of the present results to larger piece alphabets.
   
\begin{conjecture}[Uniqueness of the Bishop-King Proportionality]
\label{conj:bishop-king-unique}
Let $\Pieces' = \Pieces \cup \{P_1, \ldots, P_k\}$ be a finite extension of $\Pieces$ by additional pieces with fixed finite move sets in $\Z^2$, treated as new basic atoms of the alphabet.
Then the bishop-king proportionality $T_\B(n) = (n/12) \, T_\K(n)$ is the only relation of the form $T_P(n) = \rho(n) \, T_Q(n)$ with $P, Q \in \Pieces'$ distinct basic pieces 
and $\rho \in \mathbb{Q}[n]$ a nonconstant linear polynomial.
\end{conjecture}

For the four basic pieces $\{\K, \N, \B, \R\}$, this is a calculation: among the twelve ratios $T_P(n)/T_Q(n)$ with $P \neq Q$ basic, only $T_\B/T_\K = n/12$ lies in $\mathbb{Q}[n]$, the other eleven lying in $\mathbb{Q}(n) \setminus \mathbb{Q}[n]$.
The conjecture asks whether this persists under finite extensions of the alphabet, and in particular whether the divisor structure responsible for two of the three magic boards in our classification is genuinely without analogue.

\begin{conjecture}[Triple-Coincidence Non-Occurrence under Single-Piece Extension]
\label{conj:no-triples}
Let $P$ be a single additional piece whose move total $T_P(n)$ is a polynomial in $n$ of degree at most $3$, and let $\Pieces' = \Pieces \cup \{P\}$.
For every integer $n \geq 4$, no three distinct pieces in $\Pieces'$ share the same strength on the $n \times n$ board.
\end{conjecture}

This sharpens Corollary~\ref{cor:no-triples} to single-piece extensions of the alphabet, the intuition being as follows.
A pairwise strength coincidence $\str(P) = \str(Q)$ imposes one polynomial equation in $n$; a triple coincidence $\str(P) = \str(Q) = \str(R)$ requires two such equations to share a common integer root, a codimension-$2$ condition on the coefficient space.
Within polynomial piece-systems of degree at most $3$, the specific coefficient vectors arising from move-set unions are highly constrained, and a single new piece adds only finitely many new coincidence equations, none of which are expected to share a root with an existing one.
The single-piece restriction is essential: simultaneous addition of two or more leapers can produce triples;  for instance, the alphabet $\Pieces \cup \{(0, 5), (2, 6), (4, 5)\}$ admits the triple coincidence $T_{(0, 5)}(8) = T_{(2, 6)}(8) = T_{(4, 5)}(8) = 96$ at $n = 8$.

\paragraph{Magic boards under alphabet extensions.}
The magic-board classification depends on the choice of alphabet $\Pieces$.
Adding new pieces can introduce new strength coincidences at boards outside $\{6, 8, 12\}$.
For example, on the $4 \times 4$ board the \N{} satisfies $T_\N(4) = 48$, and this value is also attained by the Wazir, the $(1, 0)$-leaper with move total $4n(n-1)$; consequently, the extended alphabet $\Pieces \cup \{\mathrm{Wazir}\}$ acquires $n = 4$ as a fourth magic board.
A similar coincidence occurs on the $5 \times 5$ board between the \B{} and the Dabbabah, the $(0, 2)$-leaper with move total $4n(n-2)$, both of which yield strength total $60$ at $n = 5$.
Which finite extensions of $\Pieces$ preserve the magic-board set $\{6, 8, 12\}$ is an open question whose answer appears to depend sensitively on the structure of the new pieces' move sets.
What is robust, however, is that for any such extension $\Pieces'$ the resulting magic-board set is finite: each pairwise coincidence equation $T_P(n) = T_Q(n)$ is a polynomial equation in $n$ with finitely many integer roots, and there are only finitely many pairs to consider.

\paragraph{The continuous limit.}
   The asymptotic dichotomy of Theorem~\ref{thm:asymptotics} can be recast in measure-theoretic terms by embedding the $n \times n$ board into the unit square $[0, 1]^2$ at scale $1/n$ and taking $n \to \infty$.
The empirical arrow measure $\mu_n$ on $[0, 1]^2 \times [0, 1]^2$ converges, after appropriate rescaling, to a hierarchical limit distinguishing the $3$-dimensional rider strata from the $2$-dimensional leaper strata.
We develop this measure-theoretic framework in a forthcoming paper currently in progress.

\paragraph{Rectangular boards.}
Our theory specializes to square boards, and the natural generalization to $m \times n$ rectangular boards already produces interesting phenomena: the move totals extend easily ($T_\R(m, n) = mn(m + n - 2)$, with similar combinatorial identities for the \B{} and other rider pieces), and the asymptotic dichotomy of Theorem~\ref{thm:asymptotics} persists, but the ratio $\str(\B) / \str(\R)$ becomes a function of the aspect ratio $m/n$, tending to $1/3$ on $n \times n$ square boards and to $0$ on rectangles $m \times n$ with $m$ fixed and $n \to \infty$, so among rectangles of comparable total area the ratio is maximized by squares.
The magic-board classification, the stable-ordering threshold, and the Strength Algebra Theorem all admit natural rectangular generalizations, but their explicit forms involve two independent parameters $m, n$ and produce a Diophantine-coincidence problem of higher dimension.

\paragraph{Other piece families.} 
 The thirteen pieces of $\Pieces$ form one possible alphabet, chosen to span standard chess and the most common fairy-chess compounds.
Many further fairy pieces appear in the chess-variants literature~\cite{Pritchard2007, Dickins1971}, including the Wazir (orthogonal one-step leaper), the Ferz (diagonal one-step leaper), the Camel ($(1,3)$-leaper), the Zebra ($(2,3)$-leaper), the Nightrider (the \N{} as a rider, sliding multiple knight-jumps in a line), the Dabbabah ($(0,2)$-leaper), and the Alfil ($(2,2)$-leaper), as well as compounds of these.
Adding  such piece to $\Pieces$ extends the framework: each new basic piece adds an atomic dimension, so the strength function on armies becomes a $\Z$-linear functional on a lattice of higher rank, and the magic-board classification becomes a Diophantine question on this larger lattice.
A systematic treatment of these extensions is left for future work.

\paragraph{The army-counting problem.} 
The total number of armies is $|\Army| = \sum_{k=0}^{7} \multiset{13}{k} = 77520$, summing multisets of size $k$ drawn from the $13$ pieces of $\Pieces$ over all sizes $k$ from $0$ to $7$, after subtracting the mandatory \K{}.
A natural question is to enumerate the \emph{appropriate} armies, namely those whose strength on the $n \times n$ board lies within some tolerance $\varepsilon$ of the standard strength $s = 175/144$.
This is a finite combinatorial problem with closed-form ingredients and may be tractable computationally.

\paragraph{Strength versus practical piece value.}
The strength function studied in this paper is not a model of practical piece value as it appears in chess play.
Classical valuations (\B{} and \N{} both worth roughly $3$ pawns, \R{} worth $5$, \Q{} worth $9$) reflect forks, pins, center control, piece coordination, the special status of the \K{}, pawn structure, tempo, and various 
other chess-theoretic factors; our strength function captures none of this, recording only how many squares a piece attacks on average from a uniformly random starting square on the empty board.
On the standard $8 \times 8$ board, the model gives $\str(\R) : \str(\N) \approx 2.67$ and $\str(\Q) : \str(\N) \approx 4.33$, both more generous than the classical ratios $5:3$ and $3:1$; and it ranks the \N{} above the \B{} on every board $4 \leq n \leq 10$, with the order reversing for $n \geq 11$.
These mismatches are unsurprising: our model is indeed one of geometry, not strategy.
We study it nonetheless because the geometric quantity is mathematically clean, admits closed forms, exhibits the structural rigidity classified above, and provides a coherent baseline against which more sophisticated models can be compared.
The Strength Algebra Theorem and the magic-board classification are statements about the geometry of the chessboard, not about chess-theoretic aspects. 

\paragraph{Beyond the empty board.}
The model in this paper is by design agnostic about chess dynamics.
The most natural extension would be to introduce a single \emph{occupancy parameter} $p \in [0, 1]$, replacing the empty board with a board on which each square other than the starting square is occupied independently with probability $p$, and redefining the strength of a piece $P$ as the expected fraction of legal $P$-moves under this random occupancy; for $p = 0$ this reduces to our $\str(P)$.
For $p > 0$, all piece strengths are reduced, since occupancy of the destination square invalidates a move; but riders are reduced more sharply than leapers, since occupancy of an intermediate square along a slide also invalidates the move, an obstruction absent for leapers.
The asymptotic rider-leaper dichotomy of Theorem~\ref{thm:asymptotics} therefore exhibits a crossover at length scale $n \sim 1/p$.
A systematic study of this extension, which would bring the model closer to practical piece valuation, is left for future work.

\bibliographystyle{plain}

\end{document}